\documentclass[11pt]{amsart}
\usepackage{amsmath,amssymb,mathrsfs,color}

\newtheorem{lemma}{Lemma}[section]
\newtheorem{theorem}[lemma]{Theorem}
\newtheorem{remark}[lemma]{Remark}

\newtheorem{coro}[lemma]{Corollary}
\newtheorem{definition}[lemma]{Definition}
\newtheorem{example}[lemma]{Example}

\def\R{{\mathbb R}}

\def\d{{\rm d}}

\allowdisplaybreaks
\numberwithin{equation}{section}

\title[Averaging on Infinite Intervals for SODE]{Averaging Principle on Infinite Intervals
for Stochastic Ordinary Differential Equations}

\author{David Cheban}
\address{D. Cheban: School of Mathematical Sciences, Dalian University of Technology, Dalian 116024, P. R. China;
State University of Moldova, Faculty of Mathematics and Informatics, Department of Mathematics,
A. Mateevich Street 60, MD--2009 Chi\c{s}in\u{a}u, Moldova}
\email{cheban@usm.md; davidcheban@yahoo.com}

\author{Zhenxin Liu}
\address{Z. Liu (Corresponding author): School of Mathematical Sciences,
Dalian University of Technology, Dalian 116024, P. R. China}
\email{zxliu@dlut.edu.cn}

\thanks{This work is partially supported by NSFC Grants 11522104, 11871132, 11925102,
and Xinghai Jieqing and DUT19TD14 funds from Dalian University of Technology.}

\date{March 25, 2020}


\begin{document}

\begin{abstract}
In contrast to existing works on stochastic averaging on finite intervals, we establish
an averaging principle on the whole real axis, i.e. the so-called
second Bogolyubov theorem, for semilinear stochastic ordinary differential equations in Hilbert space
with Poisson stable (in particular, periodic, quasi-periodic, almost periodic,
almost automorphic etc) coefficients. Under some appropriate conditions we prove that
there exists a unique recurrent solution to the original equation, which possesses the same
recurrence property as the coefficients, in a small neighborhood of the stationary solution to
the averaged equation, and this recurrent solution converges to the stationary solution of averaged equation
uniformly on the whole real axis when the time scale approaches zero.
\end{abstract}

\maketitle

\section{Introduction}

Highly oscillating systems may be ``averaged" under some suitable conditions, and the evolution of the
averaged system can reflect in some sense the dynamics of the original system. This idea of averaging dates back to
the perturbation theory developed by Clairaut, Laplace and Lagrange in the 18th century, and is made rigorous by
Krylov, Bogolyubov, Mitropolsky \cite{KB,Bo,BM} for nonlinear oscillations.
There are vast amount of works on averaging for deterministic systems which we will not
mention here. Meantime, there are also many works on averaging principle for stochastic differential equations so far, see
e.g. \cite{CF,CL,DW,FW,Kh,Sk,Ve,Vrk_1995,WR} among others. But to our best knowledge all the existing works
on stochastic averaging are concerned with the so-called first Bogolyubov theorem, i.e. the
convergence of the solution of the original equation to that of the averaged equation on finite intervals.

In the present paper, we establish an averaging principle on the whole real axis, i.e. the so-called
second Bogolyubov theorem, for stochastic differential equations: if there exists a stationary
solution for the averaged equation, then there exists in a small neighborhood
(in the super-norm topology) a solution of the original equation which is defined on the whole
axis and has the same recurrence property (in distribution sense) as the coefficients
of the original equation. Furthermore, this recurrent solution is more general than the classical second
Bogolyubov theorem, which only treats the almost periodic case. Note that the work \cite{KMF_2015} studies
the averaging principle for stochastic differential equations with almost
periodic coefficients, but they only show the convergence on the finite interval,
not the super-norm topology on the whole axis.

To be more precise, we investigate the semilinear stochastic ordinary differential equation with
Poisson stable (in particular, periodic, quasi-periodic, Bohr almost
periodic, almost automorphic, Birkhoff recurrent, Levitan almost periodic, almost recurrent,
pseudo periodic, pseudo recurrent) in time coefficients. Under some suitable conditions, this equation has a unique
$L^2$-bounded solution which has the same recurrent properties as the coefficients, see \cite{CL_2017,LL} for details.
In this paper, we show that this recurrent solution converges to the unique stationary solution
of the averaged solution uniformly on the whole real axis when the time scale goes to zero.

The paper is organized as follows. In the second section we collect some known notions and facts.
Namely we present the construction of shift dynamical systems, definitions and basic properties of
Poisson stable functions, Shcherbakov's comparability method, and the existence of compatible
solutions for stochastic differential equations. In the third and fourth sections, we investigate
the averaging principle on infinite intervals for  linear and semilinear stochastic differential
equations respectively.

\section{Preliminaries}

\subsection{Shift dynamical systems}

Let $(\mathcal X,\rho)$ be a complete metric space and
$(\mathcal X,\mathbb R,\pi)$
be a {\em dynamical system (or flow)} on $\mathcal X$, i.e.
the mapping $\pi :\mathbb R\times \mathcal X\to \mathcal X$
is continuous, $\pi(0,x)=x$ and
$\pi(t+s,x)=\pi(t,\pi(s,x))$ for any $x\in \mathcal X$ and
$t,s\in\mathbb R$. The mapping $t\mapsto \pi(t,x)$ is called the {\em motion through $x$}.
Denote by $C(\mathbb R,\mathcal X)$ the space of
all continuous functions $\varphi :\mathbb R \to \mathcal X$
equipped with the distance
\begin{equation*}\label{eqD1}
d(\varphi_{1},\varphi_{2}):=\sum_{k=1}^{\infty}\frac{1}{2^k}
\frac{d_{k}(\varphi_1,\varphi_{2})}{1+d_{k}(\varphi_1,\varphi_{2})},
\end{equation*}
where
$$
d_{k}(\varphi_1,\varphi_{2})
:=\sup\limits_{|t|\le k}\rho(\varphi_1(t),\varphi_{2}(t)),
$$
which generates the compact-open topology on $C(\mathbb R,\mathcal X)$.
The space $(C(\mathbb R,\mathcal X),d)$ is a complete metric space
(see, e.g. \cite{Sel,Sch72,Sch85,sib}).

\begin{remark}\label{remCh}\rm
(i) Let $\varphi, \varphi_{n}\in C(\mathbb R,\mathcal X)$
($n\in\mathbb N$).
Then $\lim\limits_{n\to \infty}d(\varphi_{n},\varphi)=0$
if and only if $\lim\limits_{n\to\infty}\max\limits_{|t|\le l}\rho(\varphi_{n}(t),\varphi(t))=0$
for any $l>0$.

(ii) If there exists a sequence $l_n\to +\infty$ such that
$\lim\limits_{n\to \infty}\max\limits_{|t|\le
l_n}\rho(\varphi_{n}(t),\varphi(t))=0$, then $\lim\limits_{n\to
\infty}d(\varphi_{n},\varphi)=0$ and vice versa. See \cite{sib} for details.
\end{remark}

Let us now introduce two examples of shift dynamical systems which we will use later
in this paper.

\begin{example}\label{ex1} \rm
For given $\varphi\in C(\mathbb R,\mathcal X)$,
we denote by $\varphi^\tau$ the
{\em $\tau$-translation of $\varphi$}, i.e.
$\varphi^\tau(t)=\varphi(\tau+t)$ for $t\in\mathbb R$.
Let $\sigma :\mathbb R\times C(\mathbb R,\mathcal X)\to
C(\mathbb R,\mathcal X)$ be a mapping defined by equality
$\sigma(\tau,\varphi):=\varphi^{\tau}$ for $(\tau,\varphi)\in
\mathbb R\times C(\mathbb R,\mathcal X)$. Clearly $\sigma(0,\varphi)=\varphi$ and
$\sigma(\tau_1+\tau_2,\varphi)=\sigma(\tau_2,\sigma(\tau_1,\varphi))$
for $\varphi\in C(\mathbb R,\mathcal X)$ and $\tau_1,\tau_2\in\mathbb
R$. It is immediate to check (see, e.g. \cite{Ch2015,Sel,Sch72,sib})
that the mapping $\sigma :\mathbb R\times C(\mathbb R,\mathcal X)\to
C(\mathbb R,\mathcal X)$ is continuous, and consequently the triplet
$(C(\mathbb R,\mathcal X),\mathbb R,\sigma)$ is a dynamical system which is
called {\em shift dynamical system} or {\em Bebutov dynamical system}.

The {\em hull of $\varphi$}, denoted by $H(\varphi)$, is the set of all the limits
of $\varphi^{\tau_n}$ in $C(\mathbb R,\mathcal X)$, i.e.
\[
H(\varphi):=\{\psi\in C(\mathbb R,\mathcal X): \psi=\lim_{n\to\infty}
\varphi^{\tau_n} \hbox{ for some sequence } \{\tau_n\} \subset \mathbb
R\}.
\]
Note that the set $H(\varphi)$ is a closed and translation
invariant subset of $C(\mathbb R,\mathcal X)$ and consequently
it naturally defines on $H(\varphi)$ a shift dynamical system -- $(H(\varphi),\mathbb R,\sigma)$.
\end{example}

\begin{example}\label{ex2}\rm
Like in \cite{CL_2017}, we denote by $BUC(\mathbb{R}\times \mathcal X,\mathcal X)$
the space of all continuous functions $f:\mathbb{R}\times \mathcal X \to \mathcal X$ which are bounded on every
bounded subset from $\mathbb{R}\times \mathcal X$ and continuous in $t\in
\mathbb{R}$ uniformly with respect to $x$ on each bounded subset $Q$ of $\mathcal X$. We equip this space
with the topology of uniform convergence on bounded subsets
of $\mathbb{R}\times \mathcal X$, which can be generated by the following metric
\begin{equation}\label{eqDB1}
d(f,g):=\sum_{k=1}^{\infty}\frac{1}{2^k}\frac{d_{k}(f,g)}{1+d_{k}(f,g)},
\end{equation}
where
$$
d_{k}(f,g):=\sup\limits_{|t|\le k,\ x\in Q_k}\rho(f(t,x),g(t,x))
$$
with $Q_k\subset\mathcal X$ being bounded, $Q_k\subset Q_{k+1}$ and $\cup_{k\in\mathbb N}Q_k=\mathcal X$.

For given $f\in BUC(\mathbb{R}\times \mathcal X,\mathcal X)$ and $\tau\in\mathbb{R}$,
we denote by $f^{\tau}$ the {\em $\tau$-translation of $f$}, i.e. $f^{\tau}(t,x):=f(t+\tau,x)$ for
$(t,x)\in\mathbb{R}\times \mathcal X$. Note that the space $BUC(\mathbb{R}\times \mathcal X, \mathcal X)$ endowed
with the distance (\ref{eqDB1}) is a complete metric space and
invariant with respect to translations.
Now we define a mapping
$\sigma:\mathbb{R}\times BUC(\mathbb{R}\times \mathcal X,\mathcal X)
\to BUC(\mathbb{R}\times \mathcal X,\mathcal X)$, $(\tau,f)\mapsto f^\tau$.
It is clear that $\sigma(0,f)=f$ and
$\sigma(\tau_2,\sigma(\tau_1,f))=\sigma(\tau_1+\tau_2,f)$ for all
$f\in BUC(\mathbb{R}\times \mathcal X,\mathcal X)$ and
$\tau_1,\tau_2\in\mathbb{R}$. It is immediate to see (e.g. \cite[ChI]{Ch2015}) that
the mapping $\sigma$ is continuous and consequently the triplet
$(BUC(\mathbb{R}\times \mathcal X,\mathcal X),\mathbb{R} ,\sigma)$
is a dynamical system. Similar to Example \ref{ex1}, for given $f\in BUC(\mathbb{R}\times \mathcal X,\mathcal X)$,
the hull $H(f)$ is a closed and translation invariant subset of $BUC(\mathbb{R}\times \mathcal X,\mathcal X)$
and consequently it naturally defines on $H(f)$ a shift dynamical system -- $(H(f),\mathbb R,\sigma)$.
\end{example}

Denote by $BC(\mathcal X, \mathcal X)$ the space of all continuous functions $f: \mathcal X \to \mathcal X$
which are bounded on every bounded subset of $\mathcal X$ and equipped with the distance
\[
d(f,g):=\sum_{k=1}^{\infty}\frac{1}{2^k}\frac{d_{k}(f,g)}{1+d_{k}(f,g)},
~ d_{k}(f,g):=\sup_{x\in Q_k}\rho(f(x),g(x))
\]
where $Q_k$ are the same as above. Note that $(BC(\mathcal X,\mathcal X),d)$
is a complete metric space. For given $F\in BUC(\mathbb R\times \mathcal X, \mathcal X)$, define
$\mathcal F:  \mathbb R \to BC(\mathcal X, \mathcal X), t\mapsto \mathcal F (t)$ by letting
$\mathcal F(t):=F(t,\cdot): \mathcal X\to \mathcal X$. Clearly, $\mathcal F\in C(\mathbb R, BC(\mathcal X, \mathcal X))$.

\begin{remark}\label{remF1}\rm
The following statements are true:
\begin{enumerate}
\item
The mapping $h: BUC(\mathbb R\times \mathcal X,
\mathcal X)\to C(\mathbb R, BC(\mathcal X,\mathcal X))$
defined by equality $h(F):=\mathcal F$ establishes an isometry between
$BUC(\mathbb R\times\mathcal X,\mathcal X)$ and $C(\mathbb R, BC(\mathcal X, \mathcal X))$.

\item
$h(F^{\tau})=\mathcal F^{\tau}$ for any $\tau\in \mathbb R$ and
$F\in BUC(\mathbb R\times \mathcal X, \mathcal X)$,
i.e. the shift dynamical systems
$(BUC(\mathbb R\times\mathcal X,\mathcal X),\mathbb R,\sigma)$ and
$\left(C(\mathbb R, BC(\mathcal X,\mathcal X)),\mathbb R,\sigma\right)$
are (dynamically) homeomorphic.
\end{enumerate}
\end{remark}

\subsection{Poisson stable functions}

Let us recall the types of Poisson stable functions to be used in
this paper; we refer the reader to \cite{Sel,Sch72,Sch85,sib} for
further details and the relations among these types of functions.

\begin{definition} \rm
A function $\varphi\in C(\mathbb R,\mathcal X)$ is called
{\em stationary}
(respectively, {\em $\tau$-periodic}) if $\varphi(t)=\varphi(0)$
(respectively, $\varphi(t+\tau)=\varphi(t)$) for all $t\in \mathbb
R$.
\end{definition}

\begin{definition} \rm
(i) Let $\varepsilon >0$. A number $\tau \in \mathbb R$ is called {\em
$\varepsilon$-almost period} of the function
$\varphi:\mathbb R\rightarrow \mathcal X$ if
$\rho(\varphi(t+\tau),\varphi(t))<\varepsilon$ for all
$t\in\mathbb R$. Denote by $\mathcal T(\varphi,\varepsilon)$ the
set of $\varepsilon$-almost periods of $\varphi$.

(ii) A function $\varphi \in C(\mathbb R,\mathcal X)$
is said to be {\em Bohr almost periodic} if the set of $\varepsilon$-almost periods of
$\varphi$ is {\em relatively dense} for each $\varepsilon >0$,
i.e. for each $\varepsilon >0$ there exists a constant $l=l(\varepsilon)>0$
such that $\mathcal T(\varphi,\varepsilon)\cap [a,a+l]\not=\emptyset$ for all $a\in\mathbb R$.

(iii) A function $\varphi \in C(\mathbb R,\mathcal X)$ is said to be
{\em pseudo-periodic in the positive} (respectively, {\em negative})
{\em direction} if for each $\varepsilon >0$ and $l>0$ there exists a
$\varepsilon$-almost period $\tau >l$ (respectively, $\tau <-l$)
of the function $\varphi$. The function $\varphi$ is called
pseudo-periodic if it is {\em pseudo-periodic} in both directions.
\end{definition}

\begin{remark} \rm
A function $\varphi \in C(\mathbb R,\mathcal X)$
is pseudo-periodic in the
positive (respectively, negative) direction if and only if there
is a sequence $t_n\to +\infty$ (respectively, $t_n\to -\infty$)
such that $\varphi^{t_n}$ converges to $\varphi$ uniformly with respect to
$t\in \mathbb R$ as $n\to \infty$.
\end{remark}

\begin{definition} \rm
(i) A number $\tau\in\mathbb R$ is said to be {\em
$\varepsilon$-shift} of $\varphi \in C(\mathbb R,\mathcal X)$ if
$d(\varphi^{\tau},\varphi)<\varepsilon$; denote
$\mathfrak T(\varphi,\varepsilon):=\{\tau:
d(\varphi^{\tau},\varphi)<\varepsilon\}$.
A function $\varphi \in C(\mathbb R,\mathcal X)$
is called {\em almost recurrent (in the sense of Bebutov)}
if for every $\varepsilon >0$
the set $\mathfrak T(\varphi,\varepsilon)$ is relatively dense.

(ii) A function $\varphi\in C(\mathbb R,\mathcal X)$ is called
{\em Lagrange stable} if $\{\varphi^{\tau}: \tau\in \mathbb R\}$ is a relatively
compact subset of $C(\mathbb R,\mathcal X)$.

(iii) A function $\varphi \in C(\mathbb R,\mathcal X)$ is called
{\em Birkhoff recurrent} if it is almost recurrent and Lagrange stable.
\end{definition}

\begin{definition} \rm
A function $\varphi \in C(\mathbb R,\mathcal X)$ is called
{\em Poisson stable in the positive }
(respectively, {\em negative}) {\em direction} if for
every $\varepsilon >0$ and $l>0$ there exists $\tau >l$
(respectively, $\tau <-l$) such that
$d(\varphi^{\tau},\varphi)<\varepsilon$. The function $\varphi$ is
called {\em Poisson stable} if it is Poisson stable in both directions.
\end{definition}

In what follows, we denote as well $\mathcal Y$ a
complete metric space.

\begin{definition} \rm
(i) A function $\varphi\in C(\mathbb R,\mathcal X)$ is called
{\em Levitan almost periodic}
if there exists a Bohr almost periodic function
$\psi \in C(\mathbb R,\mathcal Y)$ such that for any
$\varepsilon >0$ there exists $\delta =\delta (\varepsilon)>0$
such that $\mathcal T(\psi,\delta)
\subseteq \mathfrak T(\varphi,\varepsilon)$.

(ii) A function $\varphi \in C(\mathbb R,\mathcal X)$ is said to be
{\em almost automorphic} if it is Levitan almost periodic and Lagrange
stable.
\end{definition}

\begin{remark} \rm Note that:
\begin{enumerate}
\item
Every Bohr almost periodic function is Levitan almost
pe\-ri\-o\-dic.
\item
The function $\varphi \in C(\mathbb R,\mathbb R)$ defined
by equality $$\varphi(t)=\dfrac{1}{2+\cos t +\cos \sqrt{2}t}$$ is
Levitan almost periodic, but it is not Bohr almost periodic
\cite[ChIV]{Lev-Zhi}.
\end{enumerate}
\end{remark}

\begin{definition} \rm
A function $\varphi \in C(\mathbb R,\mathcal X)$ is called {\em
quasi-periodic with the spectrum of frequencies
$\nu_1,\nu_2,\ldots,\nu_k$} if the following conditions are
fulfilled:
\begin{enumerate}
\item the numbers $\nu_1,\nu_2,\ldots,\nu_k$ are rationally
independent; \item there exists a continuous function $\Phi
:\mathbb R^{k}\to \mathcal X$ such that
$\Phi(t_1+2\pi,t_2+2\pi,\ldots,t_k+2\pi)=\Phi(t_1,t_2,\ldots,t_k)$
for all $(t_1,t_2,\ldots,t_k)\in \mathbb R^{k}$; \item
$\varphi(t)=\Phi(\nu_1 t,\nu_2 t,\ldots,\nu_k t)$ for $t\in
\mathbb R$.
\end{enumerate}
\end{definition}

Let $\varphi \in C(\mathbb R,\mathcal X)$. Denote by $\mathfrak
N_{\varphi}$ (respectively, $\mathfrak M_{\varphi}$) the family of
all sequences $\{t_n\}\subset \mathbb R$ such that $\varphi^{t_n}
\to \varphi$ (respectively, $\{\varphi^{t_n}\}$ converges) in
$C(\mathbb R,\mathcal X)$ as $n\to \infty$. We denote by $\mathfrak N_{\varphi}^{u}$
(respectively, $\mathfrak M_{\varphi}^{u}$) the family of sequences $\{t_n\}\in
\mathfrak N_{\varphi}$ (respectively, $\{t_n\}\in\mathfrak M_{\varphi}$)
such that $\varphi^{t_n}$ converges to $\varphi$ (respectively,  $\varphi^{t_n}$ converges) uniformly
with respect to $t\in\mathbb R$ as $n\to \infty$.

\begin{remark}\rm
\begin{enumerate}
\item The function $\varphi \in C(\mathbb R,\mathcal X)$ is pseudo-periodic
in the positive (respectively, negative) direction if and only if
there is a sequence $\{t_n\}\in \mathfrak N_{\varphi}^{u}$ such
that $t_n\to +\infty$ (respectively, $t_n\to -\infty$) as $n\to
\infty$.

\item Let $\varphi \in C(\mathbb R,\mathcal X)$,
$\psi \in C(\mathbb R,\mathcal Y)$
and $\mathfrak N_{\psi}^{u} \subseteq \mathfrak N_{\varphi}^{u}$.
If the function $\psi$ is pseudo-periodic in the positive
(respectively, negative) direction, then so is $\varphi$.
\end{enumerate}
\end{remark}

\begin{definition}\label{defPR}\rm (\cite{Sch68,Sch72,Sch85})
A function $\varphi \in C(\mathbb R,\mathcal X)$ is called  {\em
pseudo-recurrent}
if for any $\varepsilon >0$ and
$l\in\mathbb R$ there exists $L\ge l$ such that for any $\tau_0\in
\mathbb R$ we can find a number $\tau \in [l,L]$ satisfying
$$
\sup\limits_{|t|\le 1/\varepsilon}\rho(\varphi(t+\tau_0
+\tau),\varphi(t+\tau_0))\le \varepsilon.
$$
\end{definition}

\begin{remark}\label{remPR} \rm (\cite{Sch68,Sch72,Sch85,sib})
\begin{enumerate}
\item Every Birkhoff recurrent function is pseudo-recurrent, but
the inverse statement is not true in general.

\item If the function $\varphi \in C(\mathbb R,\mathcal X)$ is
pseudo-recurrent, then every function $\psi\in H(\varphi)$ is
pseudo-recurrent.

\item If  the function $\varphi \in C(\mathbb R,\mathcal X)$  is Lagrange
stable and  every function $\psi\in H(\varphi)$ is Poisson stable,
then $\varphi$ is pseudo-recurrent.
\end{enumerate}
\end{remark}

Finally, we remark that a Lagrange stable function is not Poisson
stable in general, but all other types of functions introduced
above are Poisson stable.

\begin{definition}\label{defF1} \rm
A function
$F\in BUC(\mathbb R\times \mathcal X,\mathcal X)$ is said to
possess the property $A$ in $t\in\mathbb R$
uniformly with respect to $x$ on every bounded subset
$Q$ of $\mathcal X$ if the motion $\sigma(\cdot,F)$ through $F$ with respect to the Bebutov dynamical
system $(BUC(\mathbb R\times \mathcal X,\mathcal X),\R,\sigma)$ possesses the property $A$. Here
the property $A$ may be stationary, periodic, Bohr/Levitan almost periodic etc.
\end{definition}

\begin{remark}\rm
Note that a function $\varphi\in C(\mathbb{R},\mathcal X)$ possesses the property $A$
if and only if the motion $\sigma(\cdot,\varphi):\mathbb{R}\to C(\mathbb{R},\mathcal X)$ through $\varphi$
with respect to the Bebutov dynamical system $(C(\mathbb{R},\mathcal X),\mathbb{R},\sigma)$ possesses this property.
\end{remark}

\subsection{Shcherbakov's comparability method by character of recurrence}

\begin{definition} \rm
A function $\varphi \in C(\mathbb R,\mathcal X)$ is said to be {\em
comparable} (respectively, {\em
strongly comparable}) {\em by character of
recurrence} with $\psi\in C(\mathbb R,\mathcal Y)$ if
$\mathfrak N_{\psi}\subseteq \mathfrak N_{\varphi}$ (respectively,
$\mathfrak M_{\psi}\subseteq \mathfrak M_{\varphi}$).
\end{definition}

\begin{theorem}\label{th1}{\rm(\cite[ChII]{Sch72}, \cite{scher75})}
The following statements hold:
\begin{enumerate}
\item $\mathfrak M_{\psi}\subseteq \mathfrak M_{\varphi}$ implies
$\mathfrak N_{\psi}\subseteq \mathfrak N_{\varphi}$, and hence
      strong comparability implies comparability.

\item  Let $\varphi \in C(\mathbb R,\mathcal X)$ be comparable by
    character of recurrence with $\psi\in C(\mathbb R,\mathcal Y)$.
 If the function $\psi$ is stationary (respectively, $\tau$-periodic,
 Levitan almost periodic, almost recurrent, Poisson stable),
 then so is $\varphi$.
 \item Let $\varphi \in C(\mathbb R,\mathcal X)$ be strongly
     comparable by character of recurrence with
     $\psi\in C(\mathbb R,\mathcal Y)$.
If the function $\psi$ is quasi-periodic with the spectrum of
frequencies $\nu_1,\nu_2,\dots,\nu_k$ (respectively, Bohr almost
periodic, almost automorphic, Birkhoff recurrent, Lagrange
stable), then so is $\varphi$.

\item Let $\varphi \in C(\mathbb R,\mathcal X)$ be strongly
    comparable by character of recurrence with
    $\psi\in C(\mathbb R,\mathcal Y)$ and $\psi$
be Lagrange stable. If $\psi$ is pseudo-periodic (respectively,
pseudo-recurrent), then so is $\varphi$.
\end{enumerate}
\end{theorem}

\begin{lemma}\label{lAP1}{\rm(\cite{CL_2017})}
Let $\varphi \in C(\mathbb R,\mathcal X)$,
$\psi \in C(\mathbb R,\mathcal Y)$. The following statements hold:
\begin{enumerate}
\item
If $\mathfrak M_{\psi}^{u} \subseteq \mathfrak
M_{\varphi}^{u}$, then
\begin{enumerate}
\item
   $\mathfrak N_{\psi}^{u} \subseteq \mathfrak N_{\varphi}^{u}$;
\item
   if the function $\psi$ is Bohr almost
   periodic, then so is $\varphi$.
\end{enumerate}
\item
If $\mathfrak N_{\psi}^{u} \subseteq \mathfrak
N_{\varphi}^{u}$ and $\psi$ is pseudo periodic, then so is $\varphi$.
\end{enumerate}
\end{lemma}

\subsection{Compatible solutions of semilinear stochastic differential
equations}

Let $\mathfrak B$ be a real separable Banach space with the
norm $|\cdot|$, and $L(\mathfrak B)$ be the Banach space of all bounded linear
operators acting on the space $\mathfrak B$ equipped with
operator norm $\|\cdot\|$.
Consider the linear homogeneous equation
\begin{equation}\label{eqLN}
\dot{x}=\mathcal A(t)x
\end{equation}
on the space $\mathfrak B$,
where $\mathcal A\in C(\mathbb R,L(\mathfrak B))$.
Denote by $U(t,\mathcal A)$ the Cauchy operator
(see, e.g. \cite{Dal}) of equation (\ref{eqLN}).

\begin{definition} \rm
Equation (\ref{eqLN}) is said to be {\em uniformly
asymptotically stable} if there are positive constants $\mathcal N$
and $\nu$ such that
\begin{equation}\label{eqLN1.1}
\left\|G_{\mathcal A}(t,\tau)\right\|\le \mathcal Ne^{-\nu (t-\tau)} \ \
\mbox{for any}\ t\ge \tau \ (t,\tau \in \mathbb R),
\end{equation}
where $G_{\mathcal A}(t,\tau):=U(t,\mathcal A)U^{-1}(\tau,\mathcal
A)$ for any $t,\tau\in\mathbb R$.
\end{definition}

If $\mathcal A \in C(\mathbb R,L(\mathfrak B))$,
then by $H(\mathcal A)$ we denote the closure in the space
$C(\mathbb R,L(\mathfrak B))$ of all translations $\{\mathcal A^{h}: h\in\mathbb R\}$, where
$\mathcal A^{h}(t):=\mathcal A(t+h)$ for $t\in\mathbb R$.
Denote by $C_{b}(\mathbb R,\mathfrak B)$ the Banach space of all
continuous and bounded mappings
$\varphi :\mathbb R\to \mathfrak B$ equipped with the norm
$\|\varphi\|_{\infty}:=\sup\{|\varphi(t)|: t\in\mathbb R\}$.
Note that if $f\in C_b(\mathbb R, \mathfrak B)$ and $\tilde{f}\in
H(f)$, then $\|\tilde{f}\|_{\infty}\le \|f\|_{\infty}$.

\begin{lemma}\label{lLN}\cite[ChIII]{Che_2009}
Suppose that equation (\ref{eqLN}) is uniformly asymptotically stable such that
inequality (\ref{eqLN1.1}) holds. Then
\begin{equation}\label{eqLN1.2}
\|G_{\tilde{\mathcal A}}(t,\tau)\|\le \mathcal Ne^{-\nu
(t-\tau)}\nonumber
\end{equation}
for any $t\ge \tau $ ($t,\tau \in \mathbb R$) and
$\tilde{\mathcal A}\in H(\mathcal A)$.
\end{lemma}

Let $(\Omega , \mathcal F,\mathbb P)$ be a probability space, and
$L^{2}(\mathbb P,\mathfrak B)$ be the space of $\mathfrak
B$-valued random variables $X$ such that
\begin{equation*}
\mathbb E|X|^2 :=\int_{\Omega}|X|^2 \d\mathbb P<\infty .
\end{equation*}
Then $L^2(\mathbb P,\mathfrak B)$ is a Banach space equipped with
the norm $\|X\|_{2}:=\left(\int_{\Omega}|X|^2  \d\mathbb P\right)^{1/2}$.

Let $\mathcal P(\mathfrak B)$ be the space of all Borel probability measures
on $\mathfrak B$ endowed with the $\beta$ metric:
$$
\beta (\mu,\nu) :=\sup\left\{ \left| \int f \d \mu - \int f\d
\nu\right|: \|f\|_{BL} \le 1 \right\}, \quad \hbox{for }\mu,\nu\in
\mathcal P(\mathfrak B),
$$
where $f$ are bounded Lipschitz continuous real-valued functions
on $\mathfrak B$ with the norm
\[
\|f\|_{BL}:= Lip(f) + \|f\|_\infty,~ \hbox{with } Lip(f):=\sup_{x\neq y}
\frac{|f(x)-f(y)|}{|x-y|},~ \|f\|_{\infty}:=\sup_{x\in \mathfrak B}|f(x)|.
\]
Recall that a sequence $\{\mu_n\}\subset \mathcal P(\mathfrak B)$ is said to
{\em weakly converges} to $\mu$ if $\int f \d\mu_n\to \int f \d\mu$
for all $f\in C_b(\mathfrak B)$, where $C_b(\mathfrak B)$ is the space of all bounded
continuous real-valued functions on $\mathfrak B$. It is well-known
(see, e.g. \cite[ChXI]{Dud_2004}) that $(\mathcal P(\mathfrak B),\beta)$
is a separable complete metric space and that a sequence
$\{\mu_n\}$ weakly converges to $\mu$ if and only if $\beta(\mu_n,
\mu) \to0$ as $n\to\infty$.

\begin{definition}\label{aad}\rm
A sequence of random variables $\{X_n\}$ is said to \emph{converge
in distribution} to the random variable $X$ if the corresponding
laws $\{\mu_n\}$ of $\{X_n\}$ weakly converge to the law $\mu$ of
$X$, i.e. $\beta(\mu_n,\mu)\to 0$.
\end{definition}

In the following, we assume that $\mathcal H$ is a real separable Hilbert space.
We still denote the norm in $\mathcal H$ by $|\cdot|$ and the operator norm
in $L(\mathcal H)$ by $\|\cdot\|$ which will not cause confusion.
Let us consider the stochastic differential equation
\begin{equation}\label{eqUAS_0}
\d X(t)=(\mathcal A(t)X(t)+F(t,X(t)))\d t + G(t,X(t))\d W(t),
\end{equation}
where $\mathcal A\in C(\mathbb R,L(\mathcal H))$
and $F,G\in C(\mathbb{R}\times \mathcal H,\mathcal H)$.
Here $W$ is a two-sided standard one-dimensional Brownian motion defined on the probability space
$(\Omega,\mathcal F,\mathbb P)$. We set $\mathcal F_{t}:=\sigma\{W(u):  u \le t\}$.

\begin{definition} \rm
An $\mathcal F_{t}$-adapted process
$\{X(t)\}_{t\in\mathbb R}$ is said to be a mild solution of
equation \eqref{eqUAS_0} on $\R$ if it satisfies the following
stochastic integral equation
$$
X(t)=G_{\mathcal A}(t,s)X(s)
+\int_{s}^{t}G_{\mathcal A}(t,\tau)F(\tau,X(\tau))\d \tau
+\int_{s}^{t}G_{\mathcal A}(t,\tau)G(\tau,X(\tau))\d W(\tau)
$$
for all $t\geq s$ and each $s\in\R$.
\end{definition}

\begin{definition}\label{defL1} \rm
We say that functions $F$ and $G$ satisfy the condition
\begin{enumerate}

\item[(C1)] if there exists a constant $M\ge 0$ such that
$|F(t,0)|\vee|G(t,0)|\le M$ for $t\in \mathbb R$;

\item[(C2)] if there exists a constant $ L\ge 0$ such that
$Lip(F)\vee Lip(G)\le L$, where $Lip(F):=\sup_{t\in\R, x\neq y}\frac{|F(t,x)-F(t,y)|}{|x-y|}$;

\item[(C3)] if $F$ and $G$ are continuous
            in $t$ uniformly with respect to $x$ on each bounded
            subset $Q\subset\mathcal H$.
\end{enumerate}
\end{definition}

\begin{remark}\label{remL10} \rm
\begin{enumerate}
\item
If $F$ and $G$ satisfy the conditions (C1)--(C3), then
$F, G\in BUC (\mathbb R\times \mathcal H, \mathcal H)$.

\item
If $F$ and $G$ satisfy (C1)--(C2) with the constants $M$ and
$L$, then every pair of functions $(\tilde{F},\tilde{G})$ in
$H(F,G):= \overline{\{(F^{\tau},G^{\tau}): \tau\in \mathbb R\}}$,
the hull of $(F,G)$, also possess the same property with the same
constants.
\end{enumerate}
\end{remark}

\begin{definition} \rm
Let $\{\varphi (t)\}_{t\in\mathbb R}$ be a mild solution of
equation \eqref{eqUAS_0}. Then $\varphi$ is called {\em
compatible} (respectively, {\em strongly compatible}) {\em in
distribution} if $\mathfrak N_{(\mathcal A,F,G)}
\subseteq \tilde{\mathfrak N}_{\varphi}$
(respectively,
$\mathfrak M_{(\mathcal A,F,G)}
\subseteq \tilde{\mathfrak M}_{\varphi}$), where
$\tilde{\mathfrak N}_{\varphi}$
(respectively, $\tilde{\mathfrak M}_{\varphi}$) means the set of
all sequences $\{t_n\}\subset\mathbb R$ such that the sequence
$\{\varphi(\cdot+t_n)\}$ converges to $\varphi(\cdot)$
(respectively, $\{\varphi(\cdot+t_n)\}$ converges) in distribution
uniformly on any compact interval.
\end{definition}

\begin{theorem}\label{t3.4.4}
Consider the equation \eqref{eqUAS_0}. Suppose that the following
conditions hold:
\begin{enumerate}
\item[(a)]
$\sup\limits_{t\in \mathbb R}\|\mathcal{A}(t)\|<+\infty$;
\item[(b)]
equation (\ref{eqLN}) is uniformly asymptotically stable such that \eqref{eqLN1.1} holds;
\item[(c)]
the functions $F$ and $G$ satisfy the conditions (C1) and
(C2).
\end{enumerate}

Then the following
statements hold:
\begin{enumerate}
\item
If $L< \frac{\nu}{\mathcal N \sqrt{2+\nu}}$, then
equation $(\ref{eqUAS_0})$ has a unique solution
$\xi \in C(\mathbb{R},B[0,r])$ which satisfies
\begin{equation}\label{eq b}
\xi(t)=\int_{-\infty}^{t}G_{\mathcal A}(t,\tau)F(\tau,\xi(\tau))\d\tau
+\int_{-\infty}^{t}G_{\mathcal A}(t,\tau)F(\tau,\xi(\tau))\d W(\tau),
\end{equation}
where
\begin{equation}\label{eq_r}
r=\frac{\mathcal N M\sqrt{2+\nu}}{\nu-\mathcal N L\sqrt{2+\nu}}
\end{equation}
and
$$B[0,r]:=\{x\in L^{2}(\mathbb P,\mathcal H):\ \|x\|_{2}\le r\};$$

\item
If additionally $F$ and $G$ satisfy (C3) and
$L< \frac{\nu}{2\mathcal N \sqrt{1+\nu}}$, then
\begin{enumerate}
\item
$\mathfrak M^{u}_{(\mathcal A,F,G)}
\subseteq\tilde{\mathfrak M}^{u}_{\xi}$, where $\tilde{\mathfrak M}^{u}_{\xi}$ means the
the set of all sequences $\{t_n\}\subset\mathbb R$ such that the sequence
$\{\xi(\cdot+t_n)\}$ converges in distribution uniformly on $\mathbb R$;
\item
the solution $\xi$ is strongly
compatible in distribution.
\end{enumerate}
\end{enumerate}
\end{theorem}

\begin{proof}
The proof is analogous to Theorem 4.6 in \cite{CL_2017}.
\end{proof}

\begin{coro}\label{corL2*}
Under the conditions of \textsl{Theorem} \ref{t3.4.4} the
following statements hold:
\begin{enumerate}
\item
If the functions $\mathcal A\in C(\mathbb R,L(\mathcal H))$ and
$F,G\in$ $ C(\mathbb R\times\mathcal H,\mathcal H)$ are jointly
stationary (respectively, $\tau$-periodic, quasi-periodic with the
spectrum of frequencies $\nu_1,\ldots,\nu_k$, Bohr almost
periodic, almost automorphic, Birkhoff recurrent, Lagrange
stable, Levitan almost periodic, almost recurrent, Poisson
stable), then equation (\ref{eqUAS_0}) has a unique solution
$\varphi \in C_{b}(\mathbb R,L^2(\mathbb P,\mathcal H))$ which is
stationary (respectively, $\tau$-periodic, quasi-periodic with the
spectrum of frequencies $\nu_1,\ldots,\nu_k$, Bohr almost
periodic, almost automorphic, Birkhoff recurrent, Lagrange
stable, Levitan almost periodic, almost recurrent, Poisson stable)
in distribution;
\item
If the functions
$\mathcal A\in C(\mathbb R,L(\mathcal H))$ and
$F,G\in$ $ C(\mathbb R \times \mathcal H,\mathcal H)$ are
Lagrange stable and jointly pseudo-periodic (respectively,
pseudo-re\-cur\-rent), then equation (\ref{eqUAS_0}) has a unique
solution $\varphi \in C_{b}(\mathbb R,L^2(\mathbb P,\mathcal H))$
which is pseudo-periodic (respectively, pseudo-recurrent) in
distribution.
\end{enumerate}
\end{coro}

\begin{proof}
These statements follow from Theorems \ref{th1} and \ref{t3.4.4}
(see also Remark \ref{remF1}).
\end{proof}

\section{Averaging for linear equations}

Let $\varepsilon_{0}$ be some fixed positive number. Consider the equation
\begin{equation}\label{eqALN1}
\d X(t)=\varepsilon (\mathcal A(t)X(t)+f(t))\d t
+\sqrt{\varepsilon}g(t)\d W(t),
\end{equation}
where $A\in C(\mathbb R,L(\mathcal H))$, $f,g\in
C(\mathbb R,L^2(\mathbb P, \mathcal H))$, $0<\varepsilon \le
\varepsilon_0$ and $W$ is a two-sided standard one-dimensional Brownian motion defined on the filtered
probability space $(\Omega,\mathcal F,\mathbb P,\mathcal F_{t})$,
where $\mathcal F_{t}:=\sigma\{W(u): u\le t\}$.

\begin{definition}\label{defAL1}\rm
Let $f:\mathbb R\times (0,\varepsilon_0]\to
\mathfrak B$. Following \cite{KBK} we say that {\em $f(t;\varepsilon)$
integrally converges to $0$} if for any $L>0$ we have
\begin{equation}\label{eqAL1}
\lim\limits_{\varepsilon \to 0}\sup\limits_{|t-s|\le L}
\left|\int_{s}^{t}f(\tau;\varepsilon)\d\tau\right|=0.\nonumber
\end{equation}
If additionally there exists a constant $m>0$ such that
\begin{equation}\label{eqAL2}
\left|f(t;\varepsilon)\right|\le m \ \nonumber
\end{equation}
for any $t\in\mathbb R$ and $0<\varepsilon\le \varepsilon_{0}$,
then we say that $f(t;\varepsilon)$ {\em correctly converges to} $0$ as $\varepsilon \to 0$.
\end{definition}

\begin{remark}\label{remL1.0}\cite[ChIV]{KBK} \rm
If $f\in C(\mathbb R,\mathfrak B)$
and
\begin{equation}\label{AL3}
\lim\limits_{T\to+\infty}\frac{1}{T}
\left|\int_{t}^{t+T}f(s)\d s\right|=0\nonumber
\end{equation}
uniformly with respect to $t\in\mathbb R$, then
$f(t;\varepsilon):=f(\frac{t}{\varepsilon})$ integrally converges
to $0$ as $\varepsilon \to 0$. If additionally the function $f$ is
bounded on $\mathbb R$, then $f(t;\varepsilon)$ correctly
converges to $0$ as $\varepsilon \to 0$.
\end{remark}

Let $\mathcal A\in L(\mathcal H)$.
Denote by $\sigma(\mathcal A)$ the spectrum of $\mathcal A$.
Below we will use the following conditions:
\begin{enumerate}
\item[(\textbf{A1})]
$\mathcal A\in C(\mathbb R,L(\mathcal H))$ and there
exists $\bar{\mathcal A} \in L(\mathcal H)$ such that
\begin{equation}\label{AL4.0}
\lim\limits_{T\to +\infty}\frac{1}{T}\int_{t}^{t+T}
\mathcal A(s)\d s=\bar{\mathcal A}
\end{equation}
uniformly with respect to $t\in\mathbb R$;

\item[(\textbf{A2})]
$f\in C(\mathbb R, L^2(\mathbb P, \mathcal H)$ and
there exists $\bar{f}\in L^2(\mathbb P, \mathcal H)$ such that
\begin{equation}\label{eqAL5}
\lim\limits_{T\to +\infty}\frac{1}{T} \left\|\int_{t}^{t+T}
[f(s)-\bar{f}]\d s\right\|_2=0
\end{equation}
uniformly with respect to $t\in\mathbb R$;

\item[(\textbf{A3})]
$g\in C(\mathbb R, L^2(\mathbb P, \mathcal H)$ and there exists
$\bar{g}\in  L^2(\mathbb P, \mathcal H)$ such that
\begin{equation}\label{eqAL6}
\lim\limits_{T\to
+\infty}\frac{1}{T}\int_{t}^{t+T}\mathbb E |g(\tau)-\bar{g}|^{2}\d\tau=0
\end{equation}
uniformly with respect to $t\in\mathbb R$.
\end{enumerate}

Denote by $\Psi$ the family of all decreasing, positive bounded
functions $\psi :\mathbb R_{+}\to\mathbb R_{+}$ with
$\lim\limits_{t\to +\infty}\psi(t)=0$.

\begin{lemma}\label{lC1} Let $l>0$ and $\psi\in \Psi$, then
\begin{equation*}
\lim\limits_{\varepsilon \to 0}\sup\limits_{0\le \tau\le
l}\tau\psi(\frac{\tau}{\varepsilon})=0 .
\end{equation*}
\end{lemma}

\begin{proof} Let $\varepsilon$ and $l$ be two arbitrary positive
numbers, $\nu \in(0,1)$ and $\psi\in \Psi$, then we have
\begin{equation*}\label{eqC1_2}
\sup\limits_{0\le \tau\le l}\tau\psi(\frac{\tau}{\varepsilon})
\le\sup\limits_{0\le \tau\le\varepsilon^{\nu}}\tau\psi(\frac{\tau}{\varepsilon})
+\sup\limits_{\varepsilon^{\nu}\le \tau\le l}\tau\psi(\frac{\tau}{\varepsilon})
\le\varepsilon^{\nu}\psi(0)+l\psi(\varepsilon^{\nu -1}).
\end{equation*}
Letting $\varepsilon \to 0$ we obtain the required result.
\end{proof}

\begin{remark}\label{remA1} \rm
\begin{enumerate}
  \item
  By Lemma 2 in \cite{CD_2004} equality
(\ref{AL4.0}) holds if and only if there exists a function
$\omega \in\Psi$ satisfying
\begin{equation}\label{AL4.1}
\left\|\frac{1}{T}\int_{t}^{t+T}\mathcal
A(s)\d s-\bar{\mathcal A}\right\|\le \omega(T)\nonumber
\end{equation}
for any $T>0$ and $t\in\mathbb R$.
  \item
  Similarly equality (\ref{eqAL5}) (respectively, equality (\ref{eqAL6})) holds if and only if
there exists a function $\omega_{1} \in\Psi$ (respectively,
$\omega_{2}\in\Psi$) satisfying
\begin{equation}\label{AL5.1}
\frac{1}{T}\left\|\int_{t}^{t+T} [f(s)-\bar{f}]\d s\right\|_2
\le\omega_{1}(T)\nonumber
\end{equation}
(respectively,
$\frac{1}{T}\int_{t}^{t+T} \mathbb E |g(\tau)-\bar{g}|^{2}\d\tau \le\omega_{2}(T)$)
for any $T>0$ and $t\in\mathbb R$.
\end{enumerate}
\end{remark}

\begin{theorem}\label{thAL1}\cite[ChIV]{KBK}
Suppose that $\mathcal A\in C_{b}(\mathbb R,L(\mathfrak B))$ and
\begin{equation}\label{AL4}
\lim\limits_{T \to +\infty}\frac{1}{T}\int_{t}^{t+T}\mathcal
A(s)\d s=\bar{\mathcal A}\nonumber
\end{equation}
uniformly with respect to $t\in\mathbb R$ and the operator
$\bar{\mathcal A}$ is Hurwitz, i.e. $\mathcal{R}e\ \lambda <0$
for any $\lambda \in \sigma(\bar{\mathcal A})$.

Then the following statements hold:
\begin{enumerate}
\item there exists a positive constant $\alpha \le \varepsilon_{0}$
      such that the equation
\begin{equation}\label{eqAL5.1}
\d x(\tau)=\mathcal A_{\varepsilon }(\tau)x(\tau)\d\tau,\nonumber
\end{equation}
  where $\mathcal A_{\varepsilon }(\tau):=\mathcal
  A(\frac{\tau}{\varepsilon})$ for any $\tau\in\mathbb R$, is
  uniformly asymptotically stable for any $0<\varepsilon\le \alpha$.
  Moreover there are constants $\mathcal N>0$ and $\nu>0$ such that
\begin{equation}\label{eqAL6.0}
\|G_{\mathcal A_{\varepsilon}}(\tau,\tau_0)\|
\le \mathcal Ne^{-\nu(\tau-\tau_{0})}\nonumber
\end{equation}
for any $\tau \ge\tau_{0}$ and $0<\varepsilon \le\alpha$;
\item there exists $\gamma_{0}>0$ such that
\begin{equation}\label{eqAL6.1}
\lim\limits_{\varepsilon \to 0}
\sup\limits_{(\tau\ge \tau_{0};~\tau,\tau_{0}\in \mathbb R)}
e^{\gamma_0(\tau-\tau_{0})}\|G_{\mathcal A_{\varepsilon}}(\tau,\tau_0)
- G_{\bar{\mathcal A}}(\tau,\tau_0)\|=0\ .
\end{equation}
\end{enumerate}
\end{theorem}

\begin{remark}\label{remKBK} \rm
(i) Note that Theorem \ref{thAL1} was proved for finite-dimensional
almost periodic equations (this means that the matrix-function
$\mathcal A(\cdot)$ is almost periodic). For the proof for infinite-dimensional
almost periodic systems see \cite[ChXI]{Lev-Zhi}.

(ii) It is not difficult to show that Theorem \ref{thAL1} remains
true in general case (see above) and can be proved with slight
modifications of the reasoning from \cite[ChIV]{KBK}.

(iii) Under the conditions of Theorem \ref{thAL1} there are positive
constants $\alpha, \mathcal N$ and $\nu$ so that
\begin{equation}\label{uab}
\|G_{\mathcal A_{\varepsilon}}(t,\tau)\|,
\|G_{\bar{\mathcal A}}(t,\tau)\|\le \mathcal Ne^{-\nu (t-\tau)}
\end{equation}
for any $0<\varepsilon \le \alpha$ and $t\ge\tau$.
\end{remark}

\begin{lemma}\label{lG1}
Let $f_{\varepsilon}\in C(\mathbb R, \mathfrak B)$ for
$\varepsilon \in (0,\alpha]$ be functions satisfying
the following conditions:
\begin{enumerate}
\item there exists a positive constant $A$ such that
$|f_{\varepsilon}(t)|\le A$ for any $t\in\mathbb R$ and
$\varepsilon\in (0,\alpha]$;
\item
for any $l>0$
\begin{equation}\label{eqlG1}
\lim\limits_{\varepsilon \to 0}
\sup\limits_{|s|\le l,\ t\in\mathbb R}
\left|\int_{t}^{t+s}f_{\varepsilon}(\sigma)\d\sigma\right|=0.
\end{equation}
\end{enumerate}

Then for any $\nu >0$ we have
\begin{equation}\label{eqlG2}
\lim\limits_{\varepsilon \to 0}\sup\limits_{\ t\in\mathbb R}
\left|\int_{-\infty}^{t}e^{-\nu(t-\tau)}f_{\varepsilon}(\tau)
\d\tau\right|=0.\nonumber
\end{equation}
\end{lemma}

\begin{proof}
To estimate the integral
$$
I(t,\varepsilon):=\left|\int_{-\infty}^{t}e^{-\nu
(t-\tau)}f_{\varepsilon}(\tau)\d\tau\right|
$$
we make the change $\tau -t=s$, then
\begin{equation}\label{eqlG3}
\int_{-\infty}^{t}e^{-\nu (t-\tau)}f_{\varepsilon}(\tau)\d\tau
=\int_{-\infty}^{0}e^{\nu s}f_{\varepsilon}(t+s)\d s
=\int_{-\infty}^{0}e^{\nu s}\frac{\d~}{\d s}\left(
\int_{t}^{t+s}f_{\varepsilon}(\sigma)\d\sigma\right) \d s .
\end{equation}

Since
$$
\left|e^{\nu s}\int_{t}^{t+s}f_{\varepsilon}(\sigma)\d\sigma\right|
\le Ae^{\nu s}|s|
$$
for any $s<0$, we have
\begin{equation}\label{eqlG3.1}
\lim\limits_{s\to -\infty}e^{\nu
s}\int_{t}^{t+s}f_{\varepsilon}(\sigma)\d\sigma =0\ .
\end{equation}
Integrating by parts and taking into consideration (\ref{eqlG3.1})
we obtain
\begin{align}\label{eqlG4}
\int_{-\infty}^{0}e^{\nu s}\frac{\d~}{\d s}\left(
\int_{t}^{t+s}f_{\varepsilon}(\sigma)\d\sigma\right)\d s
&
= - \int_{-\infty}^{0}\nu e^{\nu s}
\int_{t}^{t+s}f_{\varepsilon}(\sigma)\d\sigma \d s  \nonumber \\
&
=- \int_{-\infty}^{-l}\nu e^{\nu s}
\int_{t}^{t+s}f_{\varepsilon}(\sigma)\d\sigma \d s
-\int_{-l}^{0}\nu e^{\nu s}
\int_{t}^{t+s}f_{\varepsilon}(\sigma)\d\sigma \d s\ .
\end{align}
Note that
\begin{equation}\label{eqlG5}
\left|- \int_{-\infty}^{-l}\nu e^{\nu s}
\int_{t}^{t+s}f_{\varepsilon}(\sigma)\d\sigma \d s\right|
\le Ae^{-\nu l}\left(l+\frac{1}{\nu}\right)
\end{equation}
and
\begin{equation}\label{eqlG6}
\left|- \int_{-l}^{0}\nu e^{\nu s}
\int_{t}^{t+s}f_{\varepsilon}(\sigma)\d\sigma \d s\right|
\le (1-e^{-\nu l})\sup\limits_{|s|\le l,\ t\in\mathbb R}
\left|\int_{t}^{t+s}f_{\varepsilon}(\sigma)\d\sigma\right|.
\end{equation}

By (\ref{eqlG3})--(\ref{eqlG6}) we get
\begin{equation*}\label{eqlG7}
\left|\int_{-\infty}^{t}e^{-\nu (t-\tau)}f_{\varepsilon}(\tau)\d\tau\right|
\le Ae^{-\nu l}\left(l+\frac{1}{\nu}\right)
+ (1-e^{-\nu l})\sup\limits_{|s|\le l,\ t\in\mathbb R}
\left|\int_{t}^{t+s}f_{\varepsilon}(\sigma)\d\sigma\right|.
\end{equation*}
Then
\begin{equation}\label{eqlG9}
\sup\limits_{t\in\mathbb R}I(t,\varepsilon)
\le Ae^{-\nu l}\left(l+\frac{1}{\nu}\right) + (1-e^{-\nu l})
\sup\limits_{|s|\le l,\ t\in\mathbb R}
\left|\int_{t}^{t+s}f_{\varepsilon}(\sigma)\d\sigma\right| .
\end{equation}
Since
\begin{equation}\label{eqlG8}
\lim\limits_{\varepsilon \to 0}
\sup\limits_{|s|\le l,\ t\in\mathbb R}
\left|\int_{t}^{t+s}f_{\varepsilon}(\sigma)
\d\sigma\right|=0\nonumber
\end{equation}
for any $l>0$, letting  $\varepsilon \to 0$
in (\ref{eqlG9})  we have
\begin{equation*}\label{eqlG10}
\limsup\limits_{\varepsilon \to 0}\sup\limits_{t\in\mathbb
R}I(t,\varepsilon)\le  Ae^{-\nu l}\left(l+\frac{1}{\nu}\right).
\end{equation*}
Since $l$ is arbitrary, it follows that
\begin{equation}\label{eqlG11}
\lim\limits_{\varepsilon \to 0}\sup\limits_{t\in\mathbb
R}I(t,\varepsilon)=0.\nonumber
\end{equation}
The proof is complete.
\end{proof}

\begin{remark}\rm
If the function $f\in C_b(\mathbb R,\mathfrak B)$ and $\bar{f}\in \mathfrak B$ are such that
\begin{equation}\label{ave}
\lim\limits_{L\to +\infty}\frac{1}{L}\int_{t}^{t+L}
[f(s)-\bar{f}]\d s=0
\end{equation}
uniformly with respect to $t\in\mathbb R$, then the function $f_{\varepsilon}(\sigma):= f(\frac{\sigma}{\varepsilon})-\bar f$
satisfies the conditions of Lemma \ref{lG1}. Indeed, note that
\begin{align*}
\int_{t}^{t+s} f_{\varepsilon}(\sigma) \d \sigma = \int_{t}^{t+s} [f(\frac{\sigma}{\varepsilon})-\bar f] \d \sigma
=s\cdot\frac{\varepsilon}{s}\int_{t/\varepsilon}^{t/\varepsilon+s/\varepsilon} [f(\tilde\sigma)-\bar f] \d \tilde\sigma,
\end{align*}
so the condition \eqref{eqlG1} of Lemma \ref{lG1} holds by \eqref{ave}. Similarly, if the function $g$ in (A3) is $L^2$-bounded,
then the function $f_{\varepsilon}(\sigma):= \mathbb E|g(\frac{\sigma}{\varepsilon})-\bar g|^2$ satisfies as well the conditions
of Lemma \ref{lG1}.
\end{remark}

Let $W_{\varepsilon}(t):=\sqrt{\varepsilon}W(\frac{t}{\varepsilon})$ for $t\in\R$. Then
$W_\varepsilon$ is also a Brownian motion with the same distribution as $W$.

\begin{theorem}\label{thAL2}
Suppose that
$\mathcal A\in C_{b}(\mathbb R,L(\mathcal H)),
\ f,g\in C_{b}(\mathbb R,L^{2}(\mathbb P,\mathcal H))$
and conditions (A1)--(A3) are fulfilled.
Suppose further that $\bar{\mathcal A}$ in (A1) is Hurwitz such that \eqref{eqAL6.1}--\eqref{uab} holds.
Then we have the following conclusions:
\begin{enumerate}
\item equation
\begin{equation}\label{eqL8*}
\d X_{\varepsilon}(t)=(\mathcal A_{\varepsilon}(t)
X_{\varepsilon}(t)+f_{\varepsilon}(t))\d t
+g_{\varepsilon}(t)\d W_{\varepsilon}(t)
\end{equation}
has a unique bounded solution $\varphi_{\varepsilon}\in
C_{b}(\mathbb R,L^{2}(\mathbb P,\mathcal H))$ defined by equality
\begin{equation}\label{eqL7.1*}
\varphi_{\varepsilon}(t)= \int_{-\infty}^{t}G_{\mathcal
A_{\varepsilon}}(t,\tau)f_{\varepsilon}(\tau)\d\tau
+\int_{-\infty}^{t}G_{\mathcal
A_{\varepsilon}}(t,\tau)g_{\varepsilon}(\tau)\d W_{\varepsilon}(\tau)\
\nonumber
\end{equation}
and it is strongly compatible in distribution (i.e. $\mathfrak M_{(\mathcal
A_{\varepsilon},f_{\varepsilon},g_{\varepsilon})}\subseteq
\tilde{\mathfrak M}_{\varphi_{\varepsilon}}$) and $\mathfrak
M^{u}_{(\mathcal
A_{\varepsilon},f_{\varepsilon},g_{\varepsilon})}\subseteq
\tilde{\mathfrak M}^{u}_{\varphi_{\varepsilon}}$, where $\mathcal
A_{\varepsilon}(t):=\mathcal A(\frac{t}{\varepsilon})$, $f_{\varepsilon}(t):=f(\frac{t}{\varepsilon})$ and
$g_{\varepsilon}(t):=g(\frac{t}{\varepsilon})$ for $t\in\mathbb R$;

\item equation
\begin{equation}\label{eqL8}
\d X_{\varepsilon}(t)=(\mathcal A_{\varepsilon}(t)
X_{\varepsilon}(t)+f_{\varepsilon}(t))\d t
+g_{\varepsilon}(t)\d W(t)\nonumber
\end{equation}
has a unique bounded solution $\phi_{\varepsilon}\in C_{b}(\mathbb
R,L^{2}(\mathbb P,\mathcal H))$ defined by equality
\begin{equation}\label{eqL7.1}
\phi_{\varepsilon}(t)= \int_{-\infty}^{t}G_{\mathcal
A_{\varepsilon}}(t,\tau)f_{\varepsilon}(\tau)\d\tau
+\int_{-\infty}^{t}G_{\mathcal
A_{\varepsilon}}(t,\tau)g_{\varepsilon}(\tau)\d W(\tau);
\end{equation}

\item
\begin{equation}\label{eqL9}
\lim\limits_{\varepsilon \to 0}\sup\limits_{t\in\mathbb R}\mathbb
E|\phi_{\varepsilon}(t)-\bar{\phi}(t)|^2=0,
\end{equation}
where $\bar{\phi}$ is the unique stationary solution of equation
\begin{equation}\label{eqL10}
\d X(t)=(\bar{\mathcal A}X(t)+\bar{f})\d t +\bar{g}\d W(t);
\end{equation}

\item for any $\varepsilon \in (0,\alpha]$ equation (\ref{eqALN1})
has a unique solution $\varphi_{\varepsilon}\in C_{b}(\mathbb
R,L^{2}(\mathbb P,\mathcal H))$ and it is strongly compatible in
distribution (i.e. $\mathfrak M_{(\mathcal A,f,g)}\subseteq
\tilde{\mathfrak M}_{\varphi_{\varepsilon}}$) and $\mathfrak
M_{(\mathcal A,f,g)}^{u}\subseteq \tilde{\mathfrak
M}_{\varphi_{\varepsilon}}^{u}$;

\item
\begin{equation}\label{eqL9.1}
\lim\limits_{\varepsilon \to 0}\sup\limits_{t\in\mathbb R}
\beta(\mathcal
L(\varphi_{\varepsilon}(\frac{t}{\varepsilon})),
\mathcal L(\bar{\phi}(t))) =0,\nonumber
\end{equation}
where $\mathcal L(X)$ denotes the law of random variable $X$.
\end{enumerate}
\end{theorem}

\begin{proof}
The first and second statements follow directly from Theorem \ref{t3.4.4}.

We now verify the third statement, i.e. the uniform convergence of
the unique bounded solution $\phi_\varepsilon$ to the unique stationary
solution $\bar\phi$ of the averaged equation.
By Theorem \ref{t3.4.4} equation (\ref{eqL10}) has a unique bounded and stationary solution
$\bar{\phi}$, which is given by the formula
\begin{equation}\label{eqL11}
\bar{\phi}(t)= \int_{-\infty}^{t}G_{\bar{\mathcal
A}}(t,\tau)\bar{f}\d\tau +\int_{-\infty}^{t}G_{\bar{\mathcal
A}}(t,\tau)\bar{g}\d W(\tau),
\end{equation}
where
$G_{\bar{\mathcal A}}(t,\tau)
=\exp\left\{\bar{\mathcal A}(t-\tau)\right\}$ for $t,\tau\in\mathbb R$. From (\ref{eqL7.1}) and
(\ref{eqL11}) we get
\begin{align*}
&
\mathbb E|\phi_{\varepsilon}(t)-\bar{\phi}(t)|^2\\
&
=\mathbb E\bigg|\int_{-\infty}^{t}
G_{\mathcal A_{\varepsilon}}(t,\tau)
f_{\varepsilon}(\tau)\d\tau
+\int_{-\infty}^{t}G_{\mathcal A_{\varepsilon}}(t,\tau)
g_{\varepsilon}(\tau)\d W(\tau)\nonumber\\
&\qquad
-\int_{-\infty}^{t}G_{\bar{\mathcal A}}(t,\tau)\bar{f}\d\tau
-\int_{-\infty}^{t}G_{\bar{\mathcal A}}(t,\tau)\bar{g}
\d W(\tau)\bigg|^{2} \nonumber \\
&
\le2\bigg(\mathbb E\left|\int_{-\infty}^{t}
[G_{\mathcal{A_{\varepsilon}}}(t,\tau)
f_{\varepsilon}(\tau)-G_{\bar{\mathcal A}}(t,\tau)\bar{f}]\d\tau\right|^2
\nonumber \\
& \quad
+\mathbb E\left|\int_{-\infty}^{t}
\left[G_{\mathcal A_{\varepsilon}}(t,\tau)g_{\varepsilon}(\tau)
-G_{\bar{\mathcal A}}(t,\tau)\bar{g}\right]\d W(\tau)\right|^2\bigg)
=:I_{1}(t,\varepsilon)+I_{2}(t,\varepsilon).\nonumber
\end{align*}
By equality (\ref{eqAL6.1}) there exists a function $\mathcal
N:(0,\alpha)\to \mathbb R_{+}$ such that $\mathcal
N(\varepsilon)\to 0$ as $\varepsilon \to 0$  and
\begin{equation}\label{eqL14.1}
\|G_{\mathcal A_{\varepsilon}}(t,\tau)-G_{\bar{\mathcal A}}(t,\tau)\|
\le \mathcal N(\varepsilon)e^{-\gamma_{0}(t-\tau)}\nonumber
\end{equation}
for any $t\geq\tau$ ($t,\tau\in\mathbb R$).

Note that
\begin{align}\label{eqL15}
I_{1}(t,\varepsilon)
&
:=2\mathbb E\left|\int_{-\infty}^{t}[G_{\mathcal
A_{\varepsilon}}(t,\tau)f_{\varepsilon}(\tau)-G_{\bar{\mathcal
A}}(t,\tau)\bar{f}]\d\tau\right|^2 \nonumber \\
&
=2\mathbb E\left|\int_{-\infty}^{t}[G_{\mathcal
A_{\varepsilon}}(t,\tau)f_{\varepsilon}(\tau)-G_{\mathcal
A_{\varepsilon}}(t,\tau)\bar{f}] + [G_{\mathcal
A_{\varepsilon}}(t,\tau)\bar{f} -G_{\bar{\mathcal
A}}(t,\tau)\bar{f}]\d\tau\right|^2 \nonumber \\
&
\le4\left(\mathbb E\left|\int_{-\infty}^{t}
G_{\mathcal A_{\varepsilon}}(t,\tau)
\left(f_{\varepsilon}(\tau)-\bar{f}\right)\d\tau\right|^{2}
+\mathbb E\left|\int_{-\infty}^{t}
\left[G_{\mathcal A_{\varepsilon}}(t,\tau)
-G_{\bar{\mathcal A}}(t,\tau)\right]\bar{f}
\d\tau\right|^2\right)  \nonumber \\
&
=:4\left(I_{11}(t,\varepsilon)+I_{12}(t,\varepsilon)\right).\nonumber
\end{align}
To estimate the integral
\begin{equation}\label{eqL15.1}
I_{11}(t,\varepsilon):=\mathbb E\left|\int_{-\infty}^{t}
G_{\mathcal A_{\varepsilon}}(t,\tau)(f_{\varepsilon}(\tau)
-\bar{f})\d\tau\right|^{2},\nonumber
\end{equation}
making the change of variable $s:=\tau -t$ we obtain
\begin{align}\label{eqL15.2}
\int_{-\infty}^{t}G_{\mathcal A_{\varepsilon}}(t,\tau)
\left(f_{\varepsilon}(\tau)- \bar{f}\right)\d\tau
&
=\int_{-\infty}^{0}G_{\mathcal A_{\varepsilon}}(t,t+s)
\left(f_{\varepsilon}(t+s)- \bar{f}\right)\d s
\nonumber \\
&
=\int_{-\infty}^{0}G_{\mathcal A_{\varepsilon}}(t,t+s)
\frac{\d~}{\d s}\left(\int_{t}^{t+s}
[f_{\varepsilon}(\sigma)-\bar{f}]\d\sigma\right) \d s .
\end{align}
Since
\begin{equation}\label{eqL15.3}
\left\|G_{\mathcal{A}_{\varepsilon}}(t,t+s)\int_{t}^{t+s}
[f_{\varepsilon}(\sigma)-\bar{f}]\d\sigma\right\|_2
\le 2\mathcal N \|f\|_{\infty}e^{\nu s}|s|
\nonumber
\end{equation}
for any $s<0$, we have
\begin{equation}\label{eqL15.4}
\lim\limits_{s\to-\infty}G_{\mathcal{A}_{\varepsilon}}(t,t+s)
\int_{t}^{t+s}[f_{\varepsilon}(\sigma)-\bar{f}]\d\sigma =0.\nonumber
\end{equation}
Consequently, integrating by parts from (\ref{eqL15.2}) we get
\begin{equation}\label{eqL15.3.1}
\int_{-\infty}^{0}G_{\mathcal A_{\varepsilon}}(t,t+s)
\frac{\d~}{\d s}\left(\int_{t}^{t+s}
[f_{\varepsilon}(\sigma)-\bar{f}]\d\sigma\right) \d s
=-\int_{-\infty}^{0} \frac{\partial G_{\mathcal{A}_{\varepsilon}}(t,t+s)}{\partial s} \left(\int_{t}^{t+s}
[f_{\varepsilon}(\sigma)-\bar{f}]\d\sigma\right)
\d s. \nonumber
\end{equation}
Note that
\begin{equation}\label{eqL15.4.1}
\frac{\partial G_{\mathcal A}(t,\tau)}{\partial \tau}
=-G_{\mathcal A}(t,\tau)\mathcal A(\tau),\nonumber
\end{equation}
so we have
\begin{equation}\label{eqL15.5}
\left\|\frac{\partial G_{\mathcal{A}_{\varepsilon}}(t,t+s)}
{\partial s}\right\|
\le \mathcal N \|\mathcal A\|_{\infty}e^{\nu s}\nonumber
\end{equation}
for any $t\in\mathbb R$ and $s<0$.

Let now $l$ be an arbitrary positive number, then we have
\begin{align*}
&
-\int_{-\infty}^{0}\frac{\partial G_{\mathcal{A}_{\varepsilon}}(t,t+s)}{\partial s} \left(\int_{t}^{t+s}
[f_{\varepsilon}(\sigma)-\bar{f}]\d\sigma\right)
\d s \\
&
=-\int_{-\infty}^{-l}\frac{\partial G_{\mathcal{A}_{\varepsilon}}(t,t+s)}{\partial s} \left(\int_{t}^{t+s}
[f_{\varepsilon}(\sigma)-\bar{f}]\d\sigma\right)
\d s
-\int_{-l}^{0}\frac{\partial G_{\mathcal{A}_{\varepsilon}}(t,t+s)}{\partial s} \left(\int_{t}^{t+s}
[f_{\varepsilon}(\sigma)-\bar{f}]\d\sigma\right)
\d s\nonumber
\end{align*}
and consequently
\begin{align*}
&
\left\|-\int_{-\infty}^{0} \frac{\partial G_{\mathcal{A}_{\varepsilon}}(t,t+s)}{\partial s} \left(\int_{t}^{t+s}
[f_{\varepsilon}(\sigma)-\bar{f}]\d\sigma\right)
\d s\right\|_2   \\
&
\le\left\|-\int_{-\infty}^{-l}\frac{\partial G_{\mathcal{A}_{\varepsilon}}(t,t+s)}{\partial s} \left(\int_{t}^{t+s}
[f_{\varepsilon}(\sigma)-\bar{f}]\d\sigma\right)
\d s\right\|_2\nonumber\\
&\quad +\left\|\int_{-l}^{0}\frac{\partial G_{\mathcal{A}_{\varepsilon}}(t,t+s)}{\partial s} \left(\int_{t}^{t+s}
[f_{\varepsilon}(\sigma)-\bar{f}]\d\sigma\right)
\d s\right\|_2
\nonumber\\
&
\le\mathcal N \|\mathcal A\|_{\infty}
\left(2\|f\|_{\infty}\left|\int_{-\infty}^{-l}se^{\nu s}\d s\right|
+\left|\int_{-l}^{0}e^{\nu s}\d s\right|
\sup\limits_{|s|\le l,~t\in\mathbb R}
\left\|\int_{t}^{t+s}[f_{\varepsilon}(\sigma)-
\bar{f}]\d\sigma\right\|_2\right)\nonumber \\
&
\le\mathcal{N} \|\mathcal A\|_{\infty}
\left(2\|f\|_{\infty}e^{-\nu l}\left(l+\frac{1}{\nu}\right)
+\frac{1}{\nu}(1-e^{-\nu l})\sup\limits_{|s|\le l,~t\in\mathbb R}
\left\|\int_{t}^{t+s}[f_{\varepsilon}(\sigma)-\bar{f}]
\d\sigma\right\|_2\right).\nonumber
\end{align*}
Letting $\varepsilon \to 0$ in above inequality we get
\begin{align*}
&
\limsup\limits_{\varepsilon \to 0}\sup\limits_{ t\in\mathbb R}
\left\|-\int_{-\infty}^{0}\frac{\partial G_{\mathcal{A}_{\varepsilon}}(t,t+s)}{\partial s} \left(\int_{t}^{t+s}
[f_{\varepsilon}(\sigma)-\bar{f}]\d\sigma\right)
\d s\right\|_2  \\
&
\le2\mathcal{N} \|\mathcal A\|_{\infty}\|f\|_{\infty}
e^{-\nu l}\left(l+\frac{1}{\nu}\right).\nonumber
\end{align*}
Since $l$ is arbitrary, we get by letting $l\to\infty$
\begin{equation}\label{eqL15.9}
 \lim\limits_{\varepsilon \to 0}\sup\limits_{ t\in\mathbb
R}I_{11}(t,\varepsilon)=0\nonumber .
\end{equation}

Note by Theorem \ref{thAL1}--(ii) that
\begin{equation*}
I_{12}(t,\varepsilon):= \mathbb E\left|\int_{-\infty}^{t}
[G_{\mathcal A_{\varepsilon}}(t,\tau)
-G_{\bar{\mathcal A}}(t,\tau)]\bar{f}\d\tau\right|^2
\le \left(\frac{\|\bar{f}\|_{2}\mathcal N(\varepsilon)}
{\gamma_0}\right)^2\to 0
\end{equation*}
as $\varepsilon \to 0$. Consequently,
\begin{equation}\label{eqL15.11}
 \lim\limits_{\varepsilon \to 0}\sup\limits_{ t\in\mathbb
R}I_{1}(t,\varepsilon)=0\nonumber .
\end{equation}

Similarly we can show that
\begin{equation}\label{eqL22}
\lim\limits_{\varepsilon \to 0}\sup\limits_{t\in\mathbb
R}I_{2}(t,\varepsilon)=0.
\end{equation}
In fact, using It\^o's isometry property, the Cauchy-Schwartz
inequality and reasoning as above we get
\begin{align}\label{eqL23}
I_{2}(t,\varepsilon)
&
=2\mathbb E\left|\int_{-\infty}^{t}
[G_{\mathcal A_{\varepsilon}}(t,\tau)g_{\varepsilon}(\tau)
-G_{\bar{\mathcal A}}(t,\tau)\bar{g}]\d W(\tau)\right|^2  \\
&
=2\mathbb E\int_{-\infty}^{t}
\left|G_{\mathcal A_{\varepsilon}}(t,\tau)g_{\varepsilon}(\tau)
-G_{\bar{\mathcal A}}(t,\tau)\bar{g}\right|^2\d\tau \nonumber\\
&
\le 4\left(\mathbb E\int_{-\infty}^{t}
\left|G_{\mathcal A_{\varepsilon}}(t,\tau)
(g_{\varepsilon}(\tau)-\bar{g})\right|^2\d\tau
+\mathbb E\int_{-\infty}^{t}\left|\left(G_{\mathcal A_{\varepsilon}}(t,\tau)
-G_{\bar{\mathcal A}}(t,\tau)\right)\bar{g}\right|^2
\d\tau\right)\nonumber\\
&
\le 4\left(\mathbb E\int_{-\infty}^{t}\mathcal N^{2} e^{-2\nu(t-\tau)}
|g_{\varepsilon}(\tau)- \bar{g}|^{2}\d\tau
+\mathbb E\int_{-\infty}^{t}\mathcal N(\varepsilon)^{2}
e^{-2\gamma_0(t-\tau)}|\bar{g}|^{2}\d\tau\right)\nonumber\\
&
=4 \left(\mathcal N^2 \int_{-\infty}^{t}e^{-2\nu(t-\tau)}
\mathbb E|g_{\varepsilon}(\tau)- \bar{g}|^2\d\tau
+(\mathcal N(\varepsilon))^2
\frac{\|\bar{g}\|_{2}^2}{2\gamma_{0}}\right).\nonumber
\end{align}
By Lemma \ref{lG1} the integral
\begin{equation}\label{eqL23.1}
\int_{-\infty}^{t}e^{-2\nu (t-\tau)}\mathbb E|g_{\varepsilon}(\tau)
-\bar{g}|^2\d\tau
\end{equation}
goes to $0$ as $\varepsilon \to 0$ uniformly with respect to $t\in\mathbb R$.

Passing to the limit in (\ref{eqL23}) and taking into account
(\ref{eqL23.1}) we obtain (\ref{eqL22}), and consequently
$\lim\limits_{\varepsilon\rightarrow0}\sup\limits_{t\in\mathbb R}
\mathbb E|\phi_{\varepsilon}(t)-\bar{\phi}(t)|^2=0$.

To prove the fourth statement we note that the function
$\varphi_{\varepsilon}(t):=\phi_{\varepsilon}(\varepsilon t)$ (for
$t\in\mathbb R$) is a bounded solution of equation
(\ref{eqALN1}) if $\phi_{\varepsilon}$ is a bounded
solution of equation (\ref{eqL8*}). The uniqueness follows from
the fact if $\varphi_{i}$ ($i=1,2$) are two different bounded
solutions of equation (\ref{eqALN1}), then $\phi_{i}(t):=\varphi_{i}(\frac{t}{\varepsilon}), t\in\R$ ($i=1,2$)
are two different bounded solutions of equation (\ref{eqL8*}),
a contradiction to the first statement. It remains to show that $\mathfrak M_{(\mathcal
A,f,g)}\subseteq \tilde{\mathfrak M}_{\varphi_{\varepsilon}}$ and
$\mathfrak M_{(\mathcal A,f,g)}^{u}\subseteq \tilde{\mathfrak
M}_{\varphi_{\varepsilon}}^{u}$. Let $\{t_n\}\in \mathfrak
M_{(\mathcal A,f,g)}$ (respectively, $\{t_n\}\in \mathfrak
M^{u}_{(\mathcal A,f,g)}$), then $\{\varepsilon t_n\}\in \mathfrak
M_{(\mathcal
A_{\varepsilon},f_{\varepsilon},g_{\varepsilon})}\subseteq
\tilde{\mathfrak M}_{\phi_{\varepsilon}}$ (respectively,
$\{\varepsilon t_n\}\in \mathfrak M^{u}_{(\mathcal
A_{\varepsilon},f_{\varepsilon},g_{\varepsilon})}\subseteq
\tilde{\mathfrak M}^{u}_{\phi_{\varepsilon}}$) by Theorem \ref{t3.4.4}.
By the relation between $\varphi_{\varepsilon}$ and $\phi_{\varepsilon}$, we have
$\{t_n\}\in \tilde{\mathfrak M}_{\varphi_{\varepsilon}}$ (respectively, $\{t_n\}\in\tilde{\mathfrak
M}^{u}_{\varphi_{\varepsilon}}$).

Now we are in the position to prove the last statement.
Since the $L^{2}$ convergence implies convergence in probability,
it follows from (\ref{eqL9}) that
\begin{equation}\label{eqL31}
\lim\limits_{\varepsilon \to 0}\sup\limits_{t\in\mathbb
R}\beta(\mathcal L (\phi_{\varepsilon}(t)),\mathcal
L(\bar{\phi}(t)))=0\nonumber .
\end{equation}
On the other hand taking into consideration that $\mathcal
L(W)=\mathcal L(W_{\varepsilon})$, we have
$\mathcal L(\varphi_{\varepsilon}(\frac{t}{\varepsilon}))=\mathcal
L(\phi_{\varepsilon}(t))$ for any $t\in \mathbb R$, and
consequently
\begin{equation}\label{eqL32}
\lim\limits_{\varepsilon \to 0}\sup\limits_{t\in\mathbb
R}\beta(\mathcal L
(\varphi_{\varepsilon}(\frac{t}{\varepsilon})),\mathcal
L(\bar{\phi}(t)))=0\nonumber .
\end{equation}
The proof is complete.
\end{proof}

\begin{coro}\label{cor 10.1}
Under the conditions of \textsl{Theorem} \ref{thAL2} the following
statements hold:
\begin{enumerate}
\item
If the functions $\mathcal A\in C(\mathbb R,L(\mathcal H))$ and
$f,g\in$ $ C_{b}(\mathbb R,$ $L^2(\mathbb
P,\mathcal H))$ are jointly stationary (respectively,
$\tau$-periodic, quasi-periodic with the spectrum of frequencies
$\nu_1,\ldots,\nu_k$, Bohr almost periodic, almost
automorphic, Birkhoff recurrent, Lagrange stable, Levitan almost
periodic, almost recurrent, Poisson stable), then equation
(\ref{eqALN1}) has a unique solution
$\varphi_{\varepsilon} \in C_{b}(\mathbb R,L^2(\mathbb
P,\mathcal H))$ which is stationary (respectively,
$\tau$-periodic, quasi-periodic with the spectrum of frequencies
$\nu_1,\ldots,\nu_k$, Bohr almost periodic, almost
automorphic, Birkhoff recurrent, Lagrange stable, Levitan almost
periodic, almost recurrent, Poisson stable) in distribution;
\item
If the functions $\mathcal A\in C(\mathbb R,L(\mathcal H))$ and
$f,g\in$ $ C_{b}(\mathbb R,$ $L^2(\mathbb P,\mathcal H))$
are Lagrange stable and jointly pseudo-periodic
(respectively, pseudo-recurrent), then equation (\ref{eqALN1}) has
a unique solution $\varphi_{\varepsilon} \in C_{b}(\mathbb
R,L^2(\mathbb P,\mathcal H))$ which is pseudo-periodic
(respectively, pseudo-recurrent) in distribution;
\item
$$
\lim\limits_{\varepsilon \to 0}\sup\limits_{t\in\mathbb
R}\beta(\mathcal
L(\varphi_{\varepsilon}(\frac{t}{\varepsilon}),\mathcal
L(\bar{\phi}(t)))=0\ .
$$
\end{enumerate}
\end{coro}

\begin{proof}
These statements follow from Theorems \ref{th1} and \ref{thAL2}.
\end{proof}

\section{Averaging principle for semilinear
stochastic differential equations}

Consider the following stochastic differential equation
\begin{equation}\label{eqG1.1}
\d X(t)=\varepsilon\left(\mathcal A(t)X(t)+F(t,X(t))\right)\d t
+\sqrt{\varepsilon}~G(t,X(t))\d W(t),
\end{equation}
where $\mathcal A\in C(\mathbb R,L(\mathcal H))$, $F,G\in C(\mathbb R\times \mathcal H,\mathcal H)$,
$0<\varepsilon \le \varepsilon_0$ and $W$ is a two-sided standard one-dimensional Brownian motion defined on the filtered
probability space $(\Omega,\mathcal F,\mathbb P,\mathcal F_{t})$,
where $\varepsilon_0$ is a small positive constant and $\mathcal F_{t}:=\sigma\{W(u): u\le t\}$.
Below we will use the following conditions:
\begin{enumerate}
\item[\textbf{(G1)}]
there exists a positive constant $M$ such that
\begin{equation*}
|F(t,0)|\vee |G(t,0)|\le M
\end{equation*}
for any $t\in\R$;
\item[\textbf{(G2)}]
there exists a positive constant $L$ such that
\begin{equation*}
|F(t,x_{1})-F(t,x_{2})|\vee |G(t,x_{1})-G(t,x_{2})|\le
L|x_{1}-x_{2}|
\end{equation*}
for any $x_1,x_2\in \mathcal H$ and $t\in\mathbb R$;
\item[\textbf{(G3)}]
there exist functions $\omega_{1}\in \Psi$
and $\bar{F}\in C(\mathcal H,\mathcal H)$ such that
\begin{equation}\label{eqG3}
\frac{1}{T}\left|\int_{t}^{t+T}[F(s,x)-\bar{F}(x)]\d s\right|
\le\omega_{1}(T)(1+|x|)\nonumber
\end{equation}
for any $T>0$, $x\in \mathcal H$ and $t\in\mathbb R$;
 \item[\textbf{(G4)}]
there exist functions $\omega_{2}\in \Psi$ and $\bar{G}\in
C(\mathcal H,\mathcal H)$ such
that
\begin{equation}\label{eqG4}
\frac{1}{T}\int_{t}^{t+T}\left|G(s,x)-\bar{G}(x)\right|^{2}\d s
\le\omega_{2}(T)(1+|x|^{2})\nonumber
\end{equation}
for any $T>0$, $x\in \mathcal H$ and $t\in\mathbb R$;
\item[\textbf{(G5)}]
$\mathcal A\in C(\mathbb R,L(\mathcal H))$
 and there exists $\bar{\mathcal A}\in L(\mathcal H)$
 such that
\begin{equation}\label{eqG5}
\lim\limits_{T\to +\infty}\frac{1}{T}\int_{t}^{t+T}\mathcal
A(s)\d s=\bar{\mathcal A}\nonumber
\end{equation}
uniformly with respect to $t\in\mathbb R$.
\end{enumerate}

\begin{remark}\label{remML0} \rm
Under the conditions $(\textbf{G1})-(\textbf{G4})$ the
functions $\bar{F}$ and $\bar{G}$ also possess the properties
$(\textbf{G1})-(\textbf{G2})$ with the same constants $M$ and $L$.
\end{remark}

We consider as well the following equations
\begin{equation}\label{eqG2.1}
\d X(t)=(\mathcal A_{\varepsilon}(t)X(t)
+F_{\varepsilon}(t,X(t)))\d t
+G_{\varepsilon}(t,X(t))\d W(t)
\end{equation}
and
\begin{equation}\label{eqG3.1}
\d X(t)=(\mathcal A_{\varepsilon}(t)X(t)
+F_{\varepsilon}(t,X(t)))\d t
+G_{\varepsilon}(t,X(t))\d W_{\varepsilon}(t),
\end{equation}
where $\mathcal A_{\varepsilon}(t)
:=\mathcal A(\frac{t}{\varepsilon})$,
$F_{\varepsilon}(t,x):=F(\frac{t}{\varepsilon},x)$ and
$G_{\varepsilon}(t,x):=G(\frac{t}{\varepsilon},x)$
for $t\in\R$, $x\in\mathcal H$ and
$\varepsilon \in(0,\varepsilon_0]$, and $\varepsilon_{0}$ is
some fixed small positive constant. Here as before
$W_{\varepsilon}(t):=\sqrt{\varepsilon}W(\frac{t}{\varepsilon})$ for $t\in\R$.
Along with equations (\ref{eqG2.1})--(\ref{eqG3.1}) we also
consider the following averaged equation
\begin{equation}\label{eqG5_1}
\d X(t)=(\bar{\mathcal A}X(t)+\bar{F}(X(t)))\d t
+\bar{G}(X(t))\d W(t).
\end{equation}

\begin{lemma}\label{cont}
Suppose $F,G\in C(\mathbb R\times \mathcal H,\mathcal H)$ and that the conditions (G1)--(G2) hold.
If $\varphi$ is an $L^2$-bounded solution (i.e. $\|\varphi\|_\infty=\sup_{t\in\mathbb R}
\mathbb E|\varphi(t)|^2<+\infty$)
of the equation
\begin{equation*}
\d X(t)=F(t,X(t))\d t+G(t,X(t))\d W(t).
\end{equation*}
then there exists a constant $C>0$, depending only on
$M,L,\|\varphi\|_\infty$, such  that
\begin{equation*}
  \mathbb E \left|\varphi(t+h)-\varphi(t)\right|^{2}\leq Ch
\end{equation*}
and
\begin{equation*}
\mathbb E\sup_{t\leq s\leq t+h}|\varphi(s)|^{2}\leq C(h^2+1)
\end{equation*}
for any $t\in\R$ and $h>0$.
\end{lemma}

\begin{proof}
Since
\begin{equation*}
\varphi(t+h)= \varphi(t) +
\int_{t}^{t+h}F(\tau,\varphi(\tau))\d\tau
+\int_{t}^{t+h}G(\tau,\varphi(\tau))\d W(\tau),
\end{equation*}
by Cauchy-Schwartz inequality and Ito's isometry property we have
\begin{align*}
\mathbb E |\varphi(t+h)-\varphi(t)|^2
&  \le 2\left(\mathbb E \left|\int_{t}^{t+h}F(\tau,\varphi(\tau))\d\tau\right|^2 +
\mathbb E \left|\int_{t}^{t+h}G(\tau,\varphi(\tau))\d W(\tau)\right|^2\right)
\\
&\le 2\left(h  \int_{t}^{t+h}\mathbb E |F(\tau,\varphi(\tau))|^2\d\tau
+ \int_{t}^{t+h} \mathbb E |G(\tau,\varphi(\tau))|^2 \d\tau \right)
\\
& \le 4\left( h \int_{t}^{t+h}(M^2+L^2\|\varphi\|_\infty^2)\d\tau +
\int_{t}^{t+h}(M^2+L^2\|\varphi\|_\infty^2)\d\tau\right)\\
& \le C h.
\end{align*}

Employing the BDG inequality (see, e.g. \cite[Theorem 4.36]{DZ} on page 114), we have
\begin{align*}
&
\mathbb E\sup_{t\leq s\leq t+h}|\varphi(s)|^{2}\\
&
\leq 3\mathbb E |\varphi(t)|^{2}
+3\mathbb E\sup_{t\leq s\leq t+h}\left|\int_{t}^{s}
F(\tau, \varphi(\tau))
\d\tau\right|^{2} +3\mathbb E\sup_{t\leq s\leq t+h} \left|\int_{t}^{s}
G(\tau, \varphi(\tau)) \d W(\tau)\right|^{2}\\
&
\leq3 \|\varphi\|_{\infty}^{2}
+3\mathbb E\sup_{t\leq s\leq t+h}
\left|\int_{t}^{s}  (M+L|\varphi(\tau)|)\d\tau\right|^{2}
+3C\mathbb E\int_{t}^{t+h}|G(\tau, \varphi(\tau))|^{2}\d\tau\\
&
\leq3\|\varphi\|_{\infty}^{2}
+3 h \int_{t}^{t+h}
2(M^{2}+L^{2}\|\varphi\|_{\infty}^{2})\d\tau
+3C \int_{t}^{t+h}2(M^{2}+L^{2}\|\varphi\|_\infty^{2})
\d\tau\\
&
\leq C(h^2+1),
\end{align*}
where $C$ denotes some positive constants which may change
from line to line.
\end{proof}

\begin{theorem}\label{thG}
Suppose that the following conditions hold:
\begin{enumerate}
\item[(a)]
     $\sup\limits_{t\in \mathbb R}\|\mathcal{A}(t)\|<+\infty$;
\item[(b)]
     the functions $\mathcal A$, $F$, $G$ satisfy
     the conditions {\rm{(G1)--(G5)}}, and the operator $\bar{\mathcal A}$ in {\rm (G5)} is Hurwitz, i.e.
     $\mathcal{R}e\ \lambda <0$ for any
     $\lambda\in\sigma(\bar{\mathcal A})$;
\item[(c)]
\begin{equation}\label{eqG5.1}
L<\frac{\nu}{\sqrt{3}\mathcal N \sqrt{2+\nu}},\nonumber
\end{equation}
where $\mathcal N$ and $\nu$ are the numbers figuring in
Remark \ref{remKBK}-(iii).
\end{enumerate}

Then there exists a positive constant $\varepsilon_1\le \varepsilon_0$ such
that for any $0<\varepsilon \le \varepsilon_{1}$
\begin{enumerate}
\item equation (\ref{eqG1.1}) has a unique solution
      $\varphi_{\varepsilon}\in C_{b}(\mathbb R,L^{2}
      (\mathbb P,\mathcal H))$ and
      $\|\varphi_{\varepsilon}\|_{\infty}\le r$, where
      $r:=\frac{\mathcal N M\sqrt{2+\nu}}{\nu-\mathcal N L\sqrt{2+\nu}}$;
\item equation (\ref{eqG2.1}) has a unique solution
      $\phi_{\varepsilon}\in C_{b}(\mathbb R,L^{2}(\mathbb P,
      \mathcal H))$ and
      $\|\phi_{\varepsilon}\|_{\infty}\le r$;
\item if additionally $F$ and $G$ satisfy (C3) and
      $L< \frac{\nu}{2\mathcal N \sqrt{1+\nu}}$, then  the solution $\varphi_{\varepsilon}$ of equation (\ref{eqG1.1})
     is strongly compatible in distribution (i.e. $\mathfrak M_{(\mathcal A, F,G)}\subseteq
    \tilde{\mathfrak M}_{\varphi_{\varepsilon}}$) and
    $\mathfrak M_{(\mathcal A, F, G)}^{u}\subseteq \tilde{\mathfrak
    M}_{\varphi_{\varepsilon}}^{u}$, recalling that $\tilde{\mathfrak
    M}^{u}_{\varphi_{\varepsilon}}$ means the set of all sequences
    $\{t_n\}$ such that $\varphi_{\varepsilon}(t+t_n)$ converges in
    distribution uniformly with respect to $t\in\mathbb R$;
\item
  \begin{equation}\label{eqG6}
\lim\limits_{\varepsilon \to 0}\sup\limits_{t\in\mathbb R}
\mathbb E|\phi_{\varepsilon}(t)-\bar{\phi}(t)|^2=0,\nonumber
\end{equation}
where $\bar{\phi}$ is the unique stationary solution of equation
(\ref{eqG5_1});
\item
\begin{equation}\label{eqG7}
\lim\limits_{\varepsilon \to 0}\sup\limits_{t\in\mathbb R}
\beta(\mathcal L(\varphi_{\varepsilon}(\frac{t}{\varepsilon}),
\mathcal L(\bar{\phi}(t)))=0.\nonumber
\end{equation}
\end{enumerate}
\end{theorem}

\begin{proof}
By Theorem \ref{thAL1} (see also Remark \ref{remKBK}-(iii))
there exist positive constants $\mathcal N, \nu$ and $\alpha$ such
that equation
\begin{equation}\label{eqG8}
\d X(t)= \mathcal{A}_{\varepsilon}(t)X(t)\d t\nonumber
\end{equation}
is uniformly asymptotically stable for any $0<\varepsilon \le
\alpha$ and
\begin{equation}\label{eqG9}
\|G_{\mathcal{A}_{\varepsilon}}(t,\tau)\|
\le \mathcal N  e^{-\nu(t-\tau)}\nonumber
\end{equation}
for any $t\ge\tau$. By Theorem \ref{thAL1}-(ii) there are $\gamma_0>0$ and
$\mathcal N:(0,\alpha)\to \mathbb R_{+}$ such that $\mathcal N
(\varepsilon)\to 0$ as $\varepsilon \to 0$ and
\begin{equation}\label{eqG9.1}
\|G_{\mathcal{A}_{\varepsilon}}(t,\tau)-G_{\bar{\mathcal{A}}}(t,\tau)\|
\le\mathcal N(\varepsilon) e^{-\gamma_{0} (t-\tau)}\nonumber
\end{equation}
for any $t\ge\tau$.

Since $Lip(F_{\varepsilon})=Lip(F)\le L$ and $Lip(G_{\varepsilon})= Lip(G)\le L$, by Theorem \ref{t3.4.4}
equation (\ref{eqG1.1}) (respectively, equation (\ref{eqG2.1})) has a unique solution
$\varphi_{\varepsilon}$ (respectively, $\phi_{\varepsilon}$) from $C_{b}(\mathbb R, L^{2}(\mathbb
P,\mathfrak B))$ with $\varphi_{\varepsilon}\in C(\mathbb R,
B[0,r])$ (respectively, $\phi_{\varepsilon}\in C(\mathbb R,
B[0,r])$), where
$$
r:=\frac{\mathcal NM\sqrt{2+\nu}}{\nu-\mathcal N L\sqrt{2+\nu}};
$$
and the solution $\varphi_{\varepsilon}$ is strongly compatible in distribution and
$\mathfrak M_{(\mathcal A,F,G)}^{u}\subseteq \tilde{\mathfrak M}_{\varphi_{\varepsilon}}^{u}$.

Let $\bar{\phi}$ be the unique stationary solution of equation
(\ref{eqG5_1}). We now estimate $\mathbb
E|\phi_{\varepsilon}(t)-\bar{\phi}(t)|^{2}$. To this end,
we note that
\begin{align}\label{eqG10.1}
\mathbb E|\phi_{\varepsilon}(t)-\bar{\phi}(t)|^{2}
&
=\mathbb E\bigg|\int_{-\infty}^{t}G_{\mathcal{A}_{\varepsilon}}(t,\tau)
F_{\varepsilon}(\tau,\phi_{\varepsilon}(\tau))\d\tau
+\int_{-\infty}^{t}G_{\mathcal{A}_{\varepsilon}}(t,\tau)
G_{\varepsilon}(\tau,\phi_{\varepsilon}(\tau))\d W(\tau)\\
&\qquad
-\int_{-\infty}^{t}G_{\bar{\mathcal{A}}}(t,\tau)
\bar{F}(\bar{\phi}(\tau))\d\tau
-\int_{-\infty}^{t}G_{\bar{\mathcal{A}}}(t,\tau)
\bar{G}(\bar{\phi}(\tau))\d W(\tau)\bigg|^{2}\nonumber \\
&
\leq 2\bigg(\mathbb E\left|\int_{-\infty}^{t}
(G_{\mathcal{A}_{\varepsilon}}(t,\tau)
F_{\varepsilon}(\tau,\phi_{\varepsilon}(\tau))
-G_{\bar{\mathcal{A}}}(t,\tau)\bar{F}(\bar{\phi}(\tau)))
\d\tau\right|^2\nonumber \\
&\qquad
+\mathbb E\left|\int_{-\infty}^{t}
(G_{\mathcal{A}_{\varepsilon}}(t,\tau)
G_{\varepsilon}(\tau,\phi_{\varepsilon}(\tau))
-G_{\bar{\mathcal{A}}}(t,\tau)\bar{G}(\bar{\phi}(\tau)))
\d W(\tau)\right|^2\bigg)\nonumber \\
&
=:2(I_{1}(t,\varepsilon)+I_{2}(t,\varepsilon)).\nonumber
\end{align}
Since
\begin{align*}
I_{1}(t,\varepsilon)
&
:=\mathbb E\left|\int_{-\infty}^{t}
\left(G_{\mathcal{A}_{\varepsilon}}(t,\tau)
F_{\varepsilon}(\tau,\phi_{\varepsilon}(\tau))
-G_{\bar{\mathcal{A}}}(t,\tau)
\bar{F}(\bar{\phi}(\tau))\right)
\d\tau\right|^2 \\
&
\leq3\bigg(\mathbb E\left|\int_{-\infty}^{t}
G_{\mathcal{A}_{\varepsilon}}(t,\tau)
(F_{\varepsilon}(\tau,\phi_{\varepsilon}(\tau))
-F_{\varepsilon}(\tau,\bar{\phi}(\tau)))\d\tau\right|^2 \nonumber \\
&\qquad
+\mathbb E\left|\int_{-\infty}^{t}(G_{\mathcal{A}_{\varepsilon}}(t,\tau)
-G_{\bar{\mathcal{A}}}(t,\tau))F_{\varepsilon}(\tau,\bar{\phi}(\tau))
\d\tau\right|^2\nonumber \\
&\qquad
+\mathbb E\left|\int_{-\infty}^{t}G_{\bar{\mathcal{A}}}(t,\tau)
[F_{\varepsilon}(\tau,\bar{\phi}(\tau))
-\bar{F}(\bar{\phi}(\tau))]
\d\tau\right|^2\bigg),\nonumber
\end{align*}
using Cauchy-Schwartz inequality we get
\begin{align}\label{eqG12}
I_{1}(t,\varepsilon)
&
\le 3\bigg(\frac{\mathcal N^2L^{2}}{\nu}
\int_{-\infty}^{t}e^{-\nu(t-\tau)}\mathbb E
|\phi_{\varepsilon}(\tau)- \bar{\phi}(\tau)|^2\d\tau \nonumber \\
&\qquad
+\frac{2\mathcal{N}(\varepsilon)^2}{\gamma_{0}}\int_{-\infty}^{t}
e^{-\gamma_{0}(t-\tau)}(M^2+L^2\|\bar{\phi}\|^2)\d\tau
\nonumber \\
&\qquad
+\mathbb E\left|\int_{-\infty}^{t}G_{\bar{\mathcal{A}}}(t,\tau)
[F_{\varepsilon}(\tau,\bar{\phi}(\tau))
-\bar{F}(\bar{\phi}(\tau))]d\tau\right|^2\bigg)  \nonumber \\
&
\le3\bigg(\frac{\mathcal N^2L^{2}}{\nu^{2}}
\sup\limits_{t\in\mathbb R}
\mathbb E\left|\phi_{\varepsilon}(t)- \bar{\phi}(t)\right|^2
+\frac{2\mathcal{N}(\varepsilon)^2}{\gamma_{0}^{2}}
(M^2+L^2\|\bar{\phi}\|^2)\nonumber \\
&\qquad
+\mathbb E\left|\int_{-\infty}^{t}G_{\bar{\mathcal{A}}}(t,\tau)
[F_{\varepsilon}(\tau,\bar{\phi}(\tau))
-\bar{F}(\bar{\phi}(\tau))]d\tau\right|^2\bigg).
\end{align}
We will show that
\begin{equation}\label{eqG12.1}
\lim\limits_{\varepsilon \to 0}\sup\limits_{t\in\mathbb R}
\mathbb E\left|\int_{-\infty}^{t}G_{\bar{\mathcal{A}}}(t,\tau)
[F_{\varepsilon}(\tau,\bar{\phi}(\tau))
-\bar{F}(\bar{\phi}(\tau))]d\tau\right|^2=0.\nonumber
\end{equation}
To this end, making the change of variable $s=\tau-t$ and
integrating by parts, we obtain for any $l>0$
\begin{align}\label{eqG13.0}
&
\mathbb E\left|\int_{-\infty}^{t}G_{\bar{\mathcal A}}(t,\tau)
[F_{\varepsilon}(\tau,\bar{\phi}(\tau))
-\bar{F}(\bar{\phi}(\tau))]\d\tau\right|^{2}\\
&
=\mathbb E\left|\int_{-\infty}^{0}G_{\bar{\mathcal A}}(t,t+s)
[F_{\varepsilon}(t+s,\bar{\phi}(t+s))
-\bar{F}(\bar{\phi}(t+s))]\d s\right|^{2}\nonumber\\
&
=\mathbb E\left|\int_{-\infty}^{0}G_{\bar{\mathcal A}}(t,t+s)
\frac{\d~}{\d s}\left(\int_{t}^{t+s}
[F_{\varepsilon}(\sigma,\bar{\phi}(\sigma))
-\bar{F}(\bar{\phi}(\sigma))]\d\sigma\right)
\d s\right|^{2}\nonumber\\
&
\leq2\mathbb E\left|-\int_{-\infty}^{0}
\frac{\partial{G_{\bar{\mathcal A}}}(t,t+s)}{\partial s}
\left(\int_{t}^{t+s}[F_{\varepsilon}(\sigma,\bar{\phi}(\sigma))
-\bar{F}(\bar{\phi}(\sigma))]\d\sigma\right)
\d s\right|^{2}\nonumber\\
&
\leq4\mathbb E\left|-\int_{-\infty}^{-l}
\frac{\partial{G_{\bar{\mathcal A}}}(t,t+s)}{\partial s}
\left(\int_{t}^{t+s}[F_{\varepsilon}(\sigma,\bar{\phi}(\sigma))
-\bar{F}(\bar{\phi}(\sigma))]\d\sigma\right)
\d s\right|^{2}\nonumber\\
&\quad
+4\mathbb E\left|-\int_{-l}^{0}
\frac{\partial{G_{\bar{\mathcal A}}}(t,t+s)}{\partial s}
\left(\int_{t}^{t+s}[F_{\varepsilon}(\sigma,\bar{\phi}(\sigma))
-\bar{F}(\bar{\phi}(\sigma))]\d\sigma\right)
\d s\right|^{2}\nonumber\\
&
\leq4\mathbb E\left(\int_{-\infty}^{-l}\left|\int_{t}^{t+s}
[F_{\varepsilon}(\sigma,\bar{\phi}(\sigma))
-\bar{F}(\bar{\phi}(\sigma))]\d\sigma\right|\mathcal N
\|\bar{\mathcal A}\|e^{\nu s}\d s\right)^{2}\nonumber\\
&\quad
+4\mathbb E\sup_{-l\leq s\leq0}\left|\int_{t}^{t+s}
\left[F_{\varepsilon}(\sigma,\bar{\phi}(\sigma))
-\bar{F}(\bar{\phi}(\sigma))\right]\d\sigma\right|^{2}
\left|\int_{-l}^{0}\mathcal N\|\bar{\mathcal A}\|
e^{\nu s}\d s\right|^{2}
=:J_{1}+J_{2}.\nonumber
\end{align}
For $J_{1}$, we have
\begin{align}\label{eqG13.1}
J_{1}
&
:=4\mathbb E\left(\int_{-\infty}^{-l}\left|\int_{t}^{t+s}
[F_{\varepsilon}(\sigma,\bar{\phi}(\sigma))
-\bar{F}(\bar{\phi}(\sigma))]\d\sigma\right|\mathcal N
\|\bar{\mathcal A}\|e^{\nu s}\d s\right)^{2}\\
&
\leq 4\mathcal N^{2}\|\bar{\mathcal A}\|^{2}\int_{-\infty}^{-l}
e^{\nu s}\d s\int_{-\infty}^{-l}\mathbb E\left|\int_{t}^{t+s}
\left[F_{\varepsilon}(\sigma,\bar{\phi}(\sigma))
-\bar{F}(\bar{\phi}(\sigma))\right]\d\sigma\right|^{2}
e^{\nu s}\d s\nonumber\\
&
\leq\frac{4\mathcal N^{2}\|\mathcal{\bar{A}}\|^{2}}{\nu}e^{-\nu l}
\int_{-\infty}^{-l}\mathbb E\left|\int_{t}^{t+s}
\left[F_{\varepsilon}(\sigma,\bar{\phi}(\sigma))
-\bar{F}(\bar{\phi}(\sigma))\right]\d\sigma\right|^{2}
e^{\nu s}\d s\nonumber\\
&
\leq\frac{4\mathcal N^{2}\|\bar{\mathcal A}\|^{2}}{\nu}e^{-\nu l}
\int_{-\infty}^{-l}s\int_{t}^{t+s}\mathbb E
\left[F_{\varepsilon}(\sigma,\bar{\phi}(\sigma))
-\bar{F}(\bar{\phi}(\sigma))\right]^{2}\d\sigma e^{\nu s}\d s\nonumber\\
&
\leq\frac{4\mathcal N^{2}\|\bar{\mathcal A}\|^{2}}{\nu}e^{-\nu l}
\int_{-\infty}^{-l}s\int_{t}^{t+s}
8\left(M^{2}+L^{2}\|\bar{\phi}\|_{\infty}^{2}\right)
\d\sigma e^{\nu s}\d s\nonumber\\
&
\leq\frac{32\mathcal N^{2}\|\bar{\mathcal A}\|^{2}}{\nu}
\left(M^{2}+L^{2}\|\bar{\phi}\|_{\infty}^{2}\right)
e^{-\nu l}\int_{-\infty}^{-l}
s^{2}e^{\nu s}\d s\nonumber\\
&
\leq \frac{32\mathcal N^{2}\|\bar{\mathcal A}\|^{2}}{\nu}
\left(M^{2}+L^{2}\|\bar{\phi}\|_{\infty}^{2}\right)
\left(\frac{l^{2}}{\nu}+\frac{2l}{\nu^{2}}
+\frac{2}{\nu^{3}}\right)e^{-2\nu l}.\nonumber
\end{align}

Divide $[0,l]$ into intervals of size $\delta$,
where $\delta>0$ is a fixed constant
depending on $\varepsilon$.
Denote an adapted process $\tilde{\phi}$ such that
$\tilde{\phi}(\sigma)=\bar{\phi}(t-k\delta)$
for any $\sigma\in(t-(k+1)\delta,t-k\delta]$.
By Lemma \ref{cont}, we have
\begin{align*}
&
\mathbb E\sup_{-l\leq s\leq0}\left|\int_{t}^{t+s}
\left[F_{\varepsilon}(\sigma,\bar{\phi}(\sigma))
-\bar{F}(\bar{\phi}(\sigma))\right]\d\sigma\right|^{2}\\
&
=\mathbb E\sup_{-l\leq s\leq0}\left|\int_{t}^{t+s}
\left[F_{\varepsilon}(\sigma,\bar{\phi}(\sigma))
-F_{\varepsilon}(\sigma,\tilde{\phi}(\sigma))
+F_{\varepsilon}(\sigma,\tilde{\phi}(\sigma))
-\bar{F}(\tilde{\phi}(\sigma))
+\bar{F}(\tilde{\phi}(\sigma))
-\bar{F}(\bar{\phi}(\sigma))\right]\d\sigma\right|^{2}\nonumber\\
&
\leq6\mathbb E\sup_{-l\leq s\leq0}\left|\int_{t}^{t+s}
L|\bar{\phi}(\sigma)-\tilde{\phi}(\sigma)|\d\sigma\right|^{2}
+3\mathbb E\sup_{-l\leq s\leq0}\left|\int_{t}^{t+s}
\left[F_{\varepsilon}(\sigma,\tilde{\phi}(\sigma))
-\bar{F}(\tilde{\phi}(\sigma))\right]\d\sigma\right|^{2}\nonumber\\
&
\leq6\mathbb E\sup_{-l\leq s\leq0}l\int_{t+s}^{t}
L^{2}|\bar{\phi}(\sigma)-\tilde{\phi}(\sigma)|^{2}\d\sigma
+3\mathbb E\sup_{-l\leq s\leq0}\left|\int_{t}^{t+s}
\left[F_{\varepsilon}(\sigma,\tilde{\phi}(\sigma))
-\bar{F}(\tilde{\phi}(\sigma))\right]\d\sigma\right|^{2}\nonumber\\
&
\leq 6L^{2}l^{2}C\delta
+3\mathbb E\sup_{-l\leq s\leq0}\left|\int_{t}^{t+s}
\left[F_{\varepsilon}(\sigma,\tilde{\phi}(\sigma))
-\bar{F}(\tilde{\phi}(\sigma))\right]\d\sigma\right|^{2}
=:6L^{2}l^{2}C\delta+i_{2}\nonumber
\end{align*}

For $i_{2}$, denote $s(\delta):=\left[\frac{|s|}{\delta}\right]$, we have
\begin{align}\label{eqG10.7}
i_{2}
&
:=3\mathbb E\sup_{-l\leq s\leq0}\left|\int_{t}^{t+s}
\left(F_{\varepsilon}(\tau,\tilde{\phi}(\tau))
-\bar{F}(\tilde{\phi}(\tau))\right)\d\tau\right|^{2} \\
&
=3\mathbb E\sup_{-l\leq s\leq0}\Big|\sum_{k=0}^{s(\delta)-1}
\int_{t-k\delta}^{t-(k+1)\delta}
\left(F_{\varepsilon}(\tau,\bar{\phi}(t-k\delta))
-\bar{F}(\bar{\phi}(t-k\delta))\right)\d\tau\nonumber\\
&\qquad
+\int_{t-s(\delta)\cdot\delta}^{t+s}
\left(F_{\varepsilon}(\tau,\bar{\phi}(t-s(\delta)\cdot\delta))
-\bar{F}(\bar{\phi}(t-s(\delta)\cdot\delta))\right)\d\tau\Big|^{2}\nonumber\\
&
\leq6\left[\frac{l}{\delta}\right]\mathbb E\sup_{-l\leq s\leq0}
\sum_{k=0}^{s(\delta)-1}\left|\int_{t-k\delta}^{t-(k+1)\delta}
\left(F_{\varepsilon}(\tau,\bar{\phi}(t-k\delta))
-\bar{F}(\bar{\phi}(t-k\delta))\right)\d\tau\right|^{2}\nonumber\\
&\quad
+6\mathbb E\sup_{-l\leq s\leq0}\left|\int_{t-s(\delta)\cdot\delta}^{t+s}
\left(F_{\varepsilon}(\tau,\bar{\phi}(t-s(\delta)\cdot\delta))
-\bar{F}(\bar{\phi}(t-s(\delta)\cdot\delta))\right)\d\tau\right|^{2}
:=i_{2}^{1}+i_{2}^{2}.\nonumber
\end{align}
For $i_{2}^{1}$, by Lemma \ref{cont} we have
\begin{align}\label{eqG10.9}
i_{2}^{1}
&
:=6\left[\frac{l}{\delta}\right]\mathbb E\sup_{-l\leq s\leq0}
\sum_{k=0}^{s(\delta)-1}
\left|\int_{t-k\delta}^{t-(k+1)\delta}
\left(F_{\varepsilon}(\tau,\bar{\phi}(t-k\delta))
-\bar{F}(\bar{\phi}(t-k\delta))\right)\d\tau\right|^{2}\\
&
\leq\frac{6l^{2}}{\delta^{2}}\mathbb E\sup_{-l\leq s\leq0}
\max_{0\leq k\leq s(\delta)-1}
\left|\int_{\frac{t-k\delta}{\varepsilon}}
^{\frac{t-(k+1)\delta}{\varepsilon}}
\left(F(\tau,\bar{\phi}(t-k\delta))
-\bar{F}(\bar{\phi}(t-k\delta))\right)
\varepsilon\d\tau\right|^{2}\nonumber\\
&
\leq \frac{12l^{2}}{\delta^{2}}\mathbb E\sup_{-l\leq s\leq0}
\max_{0\leq k\leq s(\delta)-1}\delta^{2}\omega_{1}^{2}
\left(\frac{\delta}{\varepsilon}\right)
\left(1+|\bar{\phi}(t-k\delta)|^{2}\right)\nonumber\\
&
\leq12l^{2}\left(C+Cl^2+1\right)\omega_{1}^{2}
\left(\frac{\delta}{\varepsilon}\right).\nonumber
\end{align}
For $i_{2}^{2}$, we obtain
\begin{align}\label{eqG10.8}
i_{2}^{2}
&
:=6\mathbb E\sup_{-l\leq s\leq0}\left|\int_{t-s(\delta)\cdot\delta}^{t+s}
\left(F_{\varepsilon}(\tau,\bar{\phi}(t-s(\delta)\cdot\delta))
-\bar{F}(\bar{\phi}(t-s(\delta)\cdot\delta))\right)\d\tau\right|^{2}\\
&
\leq6\mathbb E\sup_{-l\leq s\leq0}\delta\int^{t-s(\delta)\cdot\delta}_{t+s}
\left(F_{\varepsilon}(\tau,\bar{\phi}(t-s(\delta)\cdot\delta))
-\bar{F}(\bar{\phi}(t-s(\delta)\cdot\delta))\right)^{2}\d\tau\nonumber\\
&
\leq6\delta\mathbb E\sup_{-l\leq s\leq0}\int^{t-s(\delta)\cdot\delta}_{t+s}
8\left(M^{2}+L^{2}|\bar{\phi}(t-s(\delta)\cdot\delta)|^{2}\right)\d\tau\nonumber\\
&
\leq6\delta\int_{t-l}^{t}8\left(M^{2}
+L^{2}\mathbb E\sup_{\sigma\in[t-l,t]}
\|\bar{\phi}(\sigma)\|^{2}\right)\d\tau\nonumber\\
&
\leq 48\left(M^{2}+L^{2}C(l^2+1)\right)l\delta\nonumber.
\end{align}
Therefore, \eqref{eqG10.7}--\eqref{eqG10.8} imply
\begin{equation*}
i_{2}\leq 12\left(Cl^{4}+(C+1)l^{2}\right)\omega_{1}^{2}
\left(\frac{\delta}{\varepsilon}\right)
+48\left(M^{2}+L^{2}C(l^2+1)\right)l\delta.
\end{equation*}
Therefore, we have
\begin{align}\label{eqG13.3}
J_{2}
&
\leq\frac{4\mathcal N^{2}\|\bar{\mathcal A}\|^{2}}{\nu^{2}}
\left(1-e^{-\nu l}\right)^{2}
\Big[6L^{2}l^{2}C\delta+12\left(Cl^{4}+(C+1)l^{2}\right)
\omega_{1}^{2}\left(\frac{\delta}{\varepsilon}\right)\\
&\qquad
+48\left(M^{2}+L^{2}C(l^2+1)\right)
l\delta\Big].\nonumber
\end{align}

Combing \eqref{eqG13.0}, \eqref{eqG13.1} and \eqref{eqG13.3}, we have
\begin{align}\label{eqG10.10}
&
\mathbb E\left|\int_{-\infty}^{t}G_{\bar{\mathcal A}}(t,\tau)
\left[F_{\varepsilon}(\tau,\bar{\phi}(\tau))
-\bar{F}(\bar{\phi}(\tau))\right]\d\tau\right|^{2}\\
&
\leq\frac{32\mathcal N^{2}\|\bar{\mathcal A}\|^{2}}{\nu}
\left(M^{2}+L^{2}\|\bar{\phi}\|_{\infty}^{2}\right)
\left(\frac{l^{2}}{\nu}+\frac{2l}{\nu^{2}}+\frac{2}{\nu^{3}}\right)
e^{-2\nu l}\nonumber\\
&\quad
+\frac{4\mathcal N^{2}\|\bar{\mathcal A}\|^{2}}{\nu^{2}}
\left(1-e^{-\nu l}\right)^{2}
\Big[6L^{2}l^{2}C\delta
+12\left(Cl^{4}+(C+1)l^{2}\right)
\omega_{1}^{2}\left(\frac{\delta}{\varepsilon}\right)\nonumber\\
&\qquad
+48\left(M^{2}+L^{2}C(l^2+1)\right)
l\delta\Big].\nonumber
\end{align}
Taking $\delta=\sqrt{\varepsilon}$ and letting
$\varepsilon\rightarrow0$ in \eqref{eqG10.10}, we have
\begin{align*}
&
\limsup_{\varepsilon\rightarrow0}\sup_{t\in\R}
\mathbb E\left|\int_{-\infty}^{t}G_{\bar{\mathcal A}}(t,\tau)
\left[F_{\varepsilon}(\tau,\bar{\phi}(\tau))
-\bar{F}(\bar{\phi}(\tau))\right]\d\tau\right|^{2}\\
&
\leq\frac{32\mathcal N^{2}\|\bar{\mathcal A}\|^{2}}{\nu}
\left(M^{2}+L^{2}\|\bar{\phi}\|_{\infty}^{2}\right)
\left(\frac{l^{2}}{\nu}+\frac{2l}{\nu^{2}}+\frac{2}{\nu^{3}}\right)
e^{-2\nu l}.
\end{align*}
Since $l$ is arbitrary, by letting $l\to\infty$ we get
\begin{equation}\label{eqG10.11}
\lim_{\varepsilon\rightarrow0}\sup_{t\in\R}
\mathbb E\left|\int_{-\infty}^{t}G_{\bar{\mathcal A}}(t,\tau)
\left[F_{\varepsilon}(\tau,\bar{\phi}(\tau))
-\bar{F}(\bar{\phi}(\tau))\right]\d\tau\right|^{2}=0.
\end{equation}

From \eqref{eqG12} and \eqref{eqG10.11} it follows that there
exists a function $A: (0,\varepsilon_{0})\to \mathbb R_{+}$ so that
$A(\varepsilon)\to 0$ as $\varepsilon \to 0$ and
\begin{equation}\label{eqG12.12}
I_{1}(t,\varepsilon)
\le 3\frac{\mathcal N^2L^{2}}{\nu^{2}}
\sup\limits_{t\in\mathbb R}\mathbb E
|\phi_{\varepsilon}(t)- \bar{\phi}(t)|^2 +A(\varepsilon)
\end{equation}
for any $t\in\mathbb R$ and
$\varepsilon \in (0,\varepsilon_0)$.

Now we will establish a similar estimation for
$I_{2}(t,\varepsilon)$. Since
\begin{align*}
I_{2}(t,\varepsilon)
&
:=\mathbb E\left|\int_{-\infty}^{t}
(G_{\mathcal{A}_{\varepsilon}}(t,\tau)
G_{\varepsilon}(\tau,\phi_{\varepsilon}(\tau))
-G_{\bar{\mathcal{A}}}(t,\tau)\bar{G}(\bar{\phi}(\tau)))
\d W(\tau)\right|^2\nonumber \\
&
\leq3\bigg(\mathbb E\left|\int_{-\infty}^{t}
G_{\mathcal{A}_{\varepsilon}}(t,\tau)
\left(G_{\varepsilon}(\tau,\phi_{\varepsilon}(\tau))
-G_{\varepsilon}(\tau,\bar{\phi}(\tau))\right)
\d W(\tau)\right|^2 \nonumber \\
&\qquad
+\mathbb E\left|\int_{-\infty}^{t}
\left(G_{\mathcal{A}_{\varepsilon}}(t,\tau)
-G_{\bar{\mathcal{A}}}(t,\tau)\right)
G_{\varepsilon}(\tau,\bar{\phi}(\tau))
\d W(\tau)\right|^2\nonumber \\
&\qquad
+\mathbb E\left|\int_{-\infty}^{t}G_{\bar{\mathcal{A}}}(t,\tau)
[G_{\varepsilon}(\tau,\bar{\phi}(\tau))-\bar{G}(\bar{\phi}(\tau))]
\d W(\tau)\right|^2\bigg),\nonumber
\end{align*}
using It\^o's isometry property we have
\begin{align}\label{eqG13}
I_{2}(t,\varepsilon)
&
\le 3\bigg(\mathcal N^2L^{2}
\int_{-\infty}^{t}e^{-2\nu(t-\tau)}\mathbb E
|\phi_{\varepsilon}(\tau)- \bar{\phi}(\tau)|^2\d\tau  \\
&\qquad
+2\mathcal{N}(\varepsilon)^2\int_{-\infty}^{t}
e^{-2\gamma_{0}(t-\tau)}(M^2+L^2\|\bar{\phi}\|_{\infty}^2)\d\tau
\nonumber \\
& \qquad
+\mathcal N^2\int_{-\infty}^{t}e^{-2\nu(t-\tau)}\mathbb E
|G_{\varepsilon}(\tau,\bar{\phi}(\tau))
-\bar{G}(\bar{\phi}(\tau))|^2\d\tau\bigg)  \nonumber \\
&
\le3\bigg(\frac{\mathcal N^2L^{2}}{2\nu}\sup\limits_{t\in\mathbb R}
\mathbb E|\phi_{\varepsilon}(t)- \bar{\phi}(\tau)|^2
+\frac{\mathcal{N}(\varepsilon)^2}{\gamma_{0}}
(M^2+L^2\|\bar{\phi}\|_{\infty}^2)
\nonumber \\
&\qquad
+\mathcal N^2\int_{-\infty}^{t}e^{-2\nu(t-\tau)}\mathbb E
|G_{\varepsilon}(\tau,\bar{\phi}(\tau))
-\bar{G}(\bar{\phi}(\tau))|^2d\tau\bigg). \nonumber
\end{align}
Now we prove that
\begin{equation}\label{eqG15}
\lim\limits_{\varepsilon \to 0}\sup\limits_{t\in\mathbb R}
\left|\int_{-\infty}^{t}e^{-2\nu(t-\tau)}\mathbb E
\left|G_{\varepsilon}(\tau,\bar{\phi}(\tau))
-\bar{G}(\bar{\phi}(\tau))\right|^2d\tau\right|=0.\nonumber
\end{equation}
By Lemma \ref{lG1}, it suffices to show that
\begin{equation*}
\lim\limits_{\varepsilon \to 0}
\sup\limits_{|s|\leq l,~t\in\mathbb R}
\left|\int_{t}^{t+s}\mathbb E
\left|G_{\varepsilon}(\tau,\bar{\phi}(\tau))
-\bar{G}(\bar{\phi}(\tau))\right|^2d\tau\right|=0.
\end{equation*}
To this end,
denote an adapted process $\hat{\phi}$ such that
$\hat{\phi}(\sigma)=\bar{\phi}(t+k\delta)$
for any $\sigma\in[t+k\delta,t+(k+1)\delta)$.
We can assume $s>0$ without loss of generality,
then we have by Lemma \ref{cont}
\begin{align*}
&
\int_{t}^{t+s}\mathbb E
|G_{\varepsilon}(\tau,\bar{\phi}(\tau))
-\bar{G}(\bar{\phi}(\tau))|^2d\tau\\
&
\leq\int_{t}^{t+s}\mathbb E
\left|G_{\varepsilon}(\tau,\bar{\phi}(\tau))
-G_{\varepsilon}(\tau,\hat{\phi}(\tau))
+G_{\varepsilon}(\tau,\hat{\phi}(\tau))
-\bar{G}(\hat{\phi}(\tau))
+\bar{G}(\hat{\phi}(\tau))
-\bar{G}(\bar{\phi}(\tau))\right|^{2}\d\tau\nonumber\\
&
\leq3\int_{t}^{t+s}\mathbb E
\left|G_{\varepsilon}(\tau,\bar{\phi}(\tau))
-G_{\varepsilon}(\tau,\hat{\phi}(\tau))\right|^{2}\d\tau
+3\int_{t}^{t+s}\mathbb E
\left|G_{\varepsilon}(\tau,\hat{\phi}(\tau))
-\bar{G}(\hat{\phi}(\tau))\right|^{2}\d\tau\nonumber\\
&\quad
+3\int_{t}^{t+s}\mathbb E
\left|\bar{G}(\hat{\phi}(\tau))
-\bar{G}(\bar{\phi}(\tau))\right|^{2}\d\tau\nonumber\\
&
\leq6L^{2}lC\delta+3\int_{t}^{t+s}\mathbb E
\left|G_{\varepsilon}(\tau,\hat{\phi}(\tau))
-\bar{G}(\hat{\phi}(\tau))\right|^{2}\d\tau
=:6L^{2}lC\delta+3J_{3}\nonumber
\end{align*}
For $J_{3}$,  we have
\begin{align*}
J_{3}
&
:=\mathbb E\int_{t}^{t+s}
\left|G_{\varepsilon}(\tau,\hat{\phi}(\tau))
-\bar{G}(\hat{\phi}(\tau))\right|^{2}\d\tau \\
&
\leq\mathbb E\bigg(\sum_{k=0}^{s(\delta)-1}
\int_{t+k\delta}^{t+(k+1)\delta}\left|
G_{\varepsilon}(\tau,\bar{\phi}(t+k\delta))
-\bar{G}(\bar{\phi}(t+k\delta))\right|^{2}\d\tau\nonumber\\
&\qquad
+\int_{t+s(\delta)\cdot\delta}^{t+s}\left|
G_{\varepsilon}(\tau,\bar{\phi}(t+s(\delta)\cdot\delta))
-\bar{G}(\bar{\phi}(t+s(\delta)\cdot\delta))\right|^{2}\d\tau\bigg)
=:J_{3}^{1}+J_{3}^{2}.\nonumber
\end{align*}
Then
\begin{align*}
J_{3}^{1}
&
:=\mathbb E\bigg(\sum_{k=0}^{s(\delta)-1}
\int_{t+k\delta}^{t+(k+1)\delta}\left|
G_{\varepsilon}(\tau,\bar{\phi}(t+k\delta))
-\bar{G}(\bar{\phi}(t+k\delta))\right|^{2}\d\tau\bigg)\\
&
\leq\left[\frac{l}{\delta}\right]
\max_{0\leq k\leq s(\delta)-1}\mathbb E
\int_{t+k\delta}^{t+(k+1)\delta}
\left|G_{\varepsilon}(\tau,\bar{\phi}(t+k\delta))
-\bar{G}(\bar{\phi}(t+k\delta))\right|^{2}\d\tau\nonumber\\
&
=\left[\frac{l}{\delta}\right]
\max_{0\leq k\leq s(\delta)-1}\mathbb E
\int_{\frac{t+k\delta}{\varepsilon}}
^{\frac{t+(k+1)\delta}{\varepsilon}}
\left|G(\tau,\bar{\phi}(t+k\delta))
-\bar{G}(\bar{\phi}(t+k\delta))\right|^{2}\varepsilon
\d\tau\nonumber\\
&
\leq l\omega_{2}\left(\frac{\delta}{\varepsilon}\right)
\left(1+\|\bar{\phi}\|_{\infty}^{2}\right)
\end{align*}
and
\begin{align*}
J_{3}^{2}
&
:=\mathbb E\int_{t+s(\delta)\cdot\delta}^{t+s}
\left|G_{\varepsilon}(\tau,\bar{\phi}(t+s(\delta)\cdot\delta))
-\bar{G}(\bar{\phi}(t+s(\delta)\cdot\delta))\right|^{2}\d\tau  \\
&
\leq8\left(M^{2}+L^{2}\|\bar{\phi}\|_{\infty}^{2}\right)
\delta.\nonumber
\end{align*}
Therefore we have
\begin{align}\label{eqG10.20}
&
\sup_{|s|\leq l,~t\in\R}\left|\int_{t}^{t+s}\mathbb E
\left|G_{\varepsilon}(\tau,\bar{\phi}(\tau))
-\bar{G}(\bar{\phi}(\tau))\right|^{2}\d\tau\right|\\
&
\leq 6L^{2}lC\delta
+24(M^{2}+L^{2}\|\bar{\phi}\|_{\infty}^{2})\delta
+3l\omega_{2}\left(\frac{\delta}{\varepsilon}\right)
(1+\|\bar{\phi}\|_{\infty}^{2}).\nonumber
\end{align}
Taking $\delta=\sqrt{\varepsilon}$ and letting
$\varepsilon\rightarrow0$ in \eqref{eqG10.20}, we have
\begin{equation}\label{eqG10.21}
\lim_{\varepsilon\rightarrow0}
\sup_{|s|\leq l,~t\in\R}\left|\int_{t}^{t+s}\mathbb E
\left|G_{\varepsilon}(\tau,\bar{\phi}(\tau))
-\bar{G}(\bar{\phi}(\tau))\right|^{2}\d\tau\right|=0.
\end{equation}
From \eqref{eqG13} and \eqref{eqG10.21} it follows that
\begin{equation}\label{eqG16}
I_{2}(t,\varepsilon)
\le  3(\mathcal N L)^2\frac{1}{2\nu}
\sup\limits_{t\in\mathbb R}\mathbb E
|\phi_{\varepsilon}(t)-\bar{\phi}(t)|^2
+B(\varepsilon),
\end{equation}
where $B(\varepsilon)$ is some positive constant such that
$B(\varepsilon)\to 0$ as $\varepsilon \to 0$.

Combing \eqref{eqG10.1}, \eqref{eqG12.12} and
\eqref{eqG16}, we have
\begin{equation}\label{eqG17}
\left(1-3(\mathcal N L)^2
\left(\frac{2}{\nu^2}+\frac{1}{\nu}\right)\right)
\sup_{t\in\R}\mathbb E|\phi_{\varepsilon}(t)-\bar{\phi}(t)|^2
\leq 2\left(A(\varepsilon)+B(\varepsilon)\right).\nonumber
\end{equation}
Consequently
\begin{equation}\label{eqG18}
\lim\limits_{\varepsilon \to 0}\sup\limits_{t\in\mathbb R}
\mathbb E|\phi_{\varepsilon}(t)-\bar{\phi}(t)|^2=0\nonumber
\end{equation}
because
$1-3(\mathcal N L)^2\left(\frac{2}{\nu^2}+\frac{1}{\nu}\right)>0$.

To finish the proof of the theorem we note that $L^2$-convergence
implies convergence in distribution, so
\begin{equation*}
\lim\limits_{\varepsilon \to 0}\sup\limits_{t\in\mathbb R}
\beta(\mathcal L(\phi_{\varepsilon}(t)),\mathcal L(\bar{\phi}(t)))=0.
\end{equation*}
Since
$\mathcal L(\varphi_{\varepsilon}
(\frac{t}{\varepsilon}))=\mathcal L(\phi_{\varepsilon}(t))$,
we get
\begin{equation}\label{eqG20}
\lim\limits_{\varepsilon \to 0}\sup\limits_{t\in\mathbb R}
\beta(\mathcal L(\varphi_{\varepsilon}
(\frac{t}{\varepsilon})),\mathcal L(\bar{\phi}(t)))=0.\nonumber
\end{equation}
The proof is complete.
\end{proof}

\begin{coro}\label{corL3}
Under the conditions of \textsl{Theorem} \ref{thG} the following
statements hold:
\begin{enumerate}
\item
If the functions $\mathcal A\in C(\mathbb R,L(\mathcal H))$ and
$F,G\in$ $ C(\mathbb R\times\mathcal H,\mathcal H)$ are jointly
stationary (respectively, $\tau$-periodic, quasi-periodic with the
spectrum of frequencies $\nu_1,\ldots,\nu_k$, Bohr almost
periodic, almost automorphic, Birkhoff recurrent, Lagrange
stable, Levitan almost periodic, almost recurrent, Poisson
stable), then equation (\ref{eqG1.1}) has a unique solution
$\varphi_{\varepsilon} \in C_{b}(\mathbb R,L^2(\mathbb P,\mathcal H))$ which is stationary (respectively, $\tau$-periodic,
quasi-periodic with the spectrum of frequencies
$\nu_1,\ldots,\nu_k$, Bohr almost periodic, almost
automorphic, Birkhoff recurrent, Lagrange stable, Levitan almost
periodic, almost recurrent, Poisson stable) in distribution;
\item
If the functions $\mathcal A\in C(\mathbb R,L(\mathcal H))$ and
$F,G\in$ $ C(\mathbb R \times\mathcal H,\mathcal H)$ are Lagrange stable
and jointly pseudo-periodic (respectively, pseudo-recurrent), then
equation (\ref{eqUAS_0}) has a unique solution
$\varphi_{\varepsilon} \in C_{b}(\mathbb R,L^2(\mathbb P,\mathcal H))$ which is pseudo-periodic (respectively, pseudo-recurrent) in
distribution;
\item
$$
\lim\limits_{\varepsilon \to 0}\sup\limits_{t\in\mathbb R}
\beta(\mathcal L(\varphi_{\varepsilon}(\frac{t}{\varepsilon}),
\mathcal L(\bar{\phi}(t)))=0\ .
$$
\end{enumerate}
\end{coro}

\begin{proof}
This statement follows from Theorems \ref{th1} and \ref{thG} (see also Remark \ref{remF1}).
\end{proof}

\begin{remark} \rm
In the present paper, we only consider the second Bogolyubov theorem for semilinear stochastic
{\em ordinary} differential equations, i.e. the linear part $\mathcal A(\cdot)$ is bounded operator
valued. We will consider the case when $\mathcal A(\cdot)$ is an unbounded operator in future work,
which can be applied to related stochastic partial differential equations.
\end{remark}

\end{document}